\documentclass[longarticle]{elsarticle}
\usepackage[english,french]{babel}
\usepackage{lineno}
\usepackage{fullpage}
\usepackage{amssymb}
\usepackage{amsfonts}
\usepackage{amsmath}
\usepackage{mathtools} 
\usepackage{graphicx, float}
\usepackage{tikz}
\usepackage{epsfig,mathrsfs}
\usepackage[bf,SL,BF]{subfigure}
\usepackage{color}
\newcommand{\proofend}{\hfill $\Box$ \vspace{2mm}}
\newtheorem{theorem}{Theorem}
\newtheorem{assum}{Assumption}

\newtheorem{lemma}{Lemma}
\newtheorem{definition}{Definition}

\newcommand{\R}{\mathbb{R}}
\newcommand{\C}{\mathbb{C}}
\newcommand{\re}{\mbox{Re}}
\newcommand{\im}{\mbox{Im}}

\newcommand{\vecB}{{\bf B}}
\newcommand{\vecE}{{\bf E}}
\newcommand{\vecf}{{\bf f}}
\newcommand{\vecg}{{\bf g}}
\newcommand{\vecR}{{\bf R}}
\newcommand{\vecL}{{\bf L}}
\newcommand{\vecH}{{\bf H}}
\newcommand{\scH}[1]{\overline{\bf H}^{#1}_{sc}} 
\newcommand{\vecu}{{\bf u}}
\newcommand{\uvecu}{\underline{\bf u}}
\newcommand{\uvecv}{\underline{\bf v}}

\newcommand{\vecv}{  {\bf v}  }

\newcommand{\vecp}{  {\bf p}  }
\newcommand{\vecq}{  {\bf q}  }

\newcommand{\vecS}{{\bf S}}

\newcommand{\vecU}{{\bf U}}
\newcommand{\vecV}{{\bf V}}
\newcommand{\vecF}{{\bf F}}
\newcommand{\vecG}{{\bf G}}

\newcommand{\vecP}{{\bf P}}

\newcommand{\vecJ}{{\bf J}}

\newcommand{\x}{{\bf x}}

\newcommand{\mdiv}{{ \rm div\,}}
\newcommand{\mcurl}{{\rm curl \,}}
\newcommand{\mop}{\mbox{op}}
\newcommand{\mOp}{\mbox{Op}}
\graphicspath{{pics/}}

\newenvironment{abstracts}
 {\global\setbox\absbox=\vbox\bgroup
    \hsize=\textwidth
    \linespread{1}\selectfont}
 {\vspace{-\bigskipamount}\egroup}
\renewenvironment{abstract}[1][]
 {\if\relax\detokenize{#1}\relax\else\selectlanguage{#1}\fi
  \noindent\textbf{\abstractname}\par\medskip\noindent\ignorespaces}
 {\par\bigskip}

\begin{document}

\begin{frontmatter}

\title{The Spectral Analysis of the Interior Transmission Eigenvalue Problem for Maxwell's Equations}

\author
{
Houssem Haddar, Shixu Meng\fnref{author2}\corref{cor1}
}

\fntext[author2]{Current address (when article was published): Department of Mathematics,
University of Michigan, Ann Arbor 48109, USA}
\cortext[cor1]{Corresponding author}

\address
{
INRIA and Ecole Polytechnique (CMAP),  Route de Saclay, 91128 Palaiseau Cedex, France
}

\ead{Houssem.Haddar@inria.fr, shixumen@umich.edu}

\begin{abstracts}
\selectlanguage{english}
\begin{abstract}
In this paper we consider the transmission eigenvalue problem  for Maxwell's
equations corresponding to non-magnetic inhomogeneities with contrast in
electric permittivity that has fixed sign (only) in a neighborhood of the boundary.  Following the
analysis made by Robbiano in the scalar case we study this problem in the
framework of semiclassical analysis and relate the transmission eigenvalues to
the spectrum of a Hilbert-Schmidt operator.  Under the additional assumption
that the contrast is constant in a neighborhood of the boundary, we prove that
the set of transmission eigenvalues is discrete, infinite and without finite accumulation
points. A notion of generalized
eigenfunctions is introduced and a denseness result is obtained in an
appropriate solution space. 
\end{abstract} 

\selectlanguage{french}
\begin{abstract}
Nous consid\'erons le probl\`eme de valeurs propres de transmission pour les \'equations de Maxwell o\`u le contraste sur la permittivit\'e \'electrique admet un signe fixe seulement sur un voisinage du bord du domaine. En suivant l'approche d\'evelopp\'ee par Robbiano dans le cas scalaire, nous formulons le probl\`eme aux valeurs propres comme un probl\`eme spectral pour un op\'erateur d'Hilbert-Schmidt. Sous l'hypoth\`ese suppl\'ementaire que le contraste est constant au voisinage de la fronti\`ere nous montrons que l'ensemble des valeurs propres de transmission est discret sans points d'accumulation finis. Nous d\'efinissons une famille de fonctions propres g\'en\'eralis\'ees pour laquelle un r\'esultat de densit\'e est obtenu dans un espace de solutions bien choisi.
\end{abstract} 
\end{abstracts}
\selectlanguage{english}


\end{frontmatter}
\section{Introduction}  
The transmission eigenvalue problem is  related to the scattering problem for an inhomogeneous media. In the current paper the underlying scattering problem is the scattering of electromagnetic waves by a non-magnetic material of bounded support $D$ situated in homogenous background, which in terms of the electric field reads: 
\begin{eqnarray*}
&\mcurl \mcurl \vecE^s-k^2 \vecE^s = 0 \quad &\mbox{in} \quad {\mathbb R}^3\setminus\overline{D}  \\
&\mcurl \mcurl \vecE-k^2 n \vecE = 0 \quad &\mbox{in} \quad D  \\
&\nu \times \vecE = \nu \times \vecE^s+\nu \times \vecE^i \quad &\mbox{on} \quad \partial D\\ 
&\nu \times \mcurl \vecE = \nu \times \mcurl \vecE^s+\nu \times \mcurl \vecE^i \quad &\mbox{on} \quad \partial D \\
&\lim\limits_{r\to \infty}\left(\mcurl \vecE^s\times x-ikr\vecE^s\right)=0 &
\end{eqnarray*}
where $\vecE^i$ is the incident electric field, $\vecE^s$ is the scattered
electric field,
$n(x)$
is the index of refraction, $k$ is the wave number
 and the
Silver-M{\"u}ller radiation condition is satisfied uniformly with respect to
$\hat x=x/r$, $r=|x|$.  The difference  $n-1$  is refereed to as the contrast in the media. In scattering theory, transmission eigenvalues can be seen as the extension of the notion of resonant frequencies for impenetrable objects to the case of penetrable media. The transmission eigenvalue problem is related to non-scattering incident fields \cite{CaCo, CaCoMo, CK}. Indeed, if $\vecE^i$ is such that $\vecE^s=0$ then $\vecE|_{D}$ and $\vecE_0=\vecE^i|_{D}$ satisfy the following homogenous problem
\begin{eqnarray}
&\mcurl \mcurl \vecE-k^2 n \vecE = 0 \quad &\mbox{in} \quad D  \label{IntroE} \\
&\mcurl \mcurl \vecE_0-k^2 \vecE_0 = 0 \quad &\mbox{in} \quad D \label{IntroE0} \\
&\nu \times \vecE = \nu \times \vecE_0 \quad &\mbox{on} \quad \Gamma \label{IntroDirichlet} \\ 
&\nu \times \mcurl \vecE = \nu \times \mcurl \vecE_0 \quad &\mbox{on} \quad \Gamma  \label{IntroNeumann}
\end{eqnarray}
with $\Gamma:=\partial D$ and $\nu$ the inward unit normal vector on $\Gamma$, which is referred to as the transmission eigenvalue problem. If the above
problem has a non trivial solution then $k$ is called a transmission eigenvalue. Conversely, if the
above equations have a nontrivial solution $\vecE$ and $\vecE_0$, and
$\vecE_0$ can be extended outside $D$ as a solution to $\mcurl \mcurl
\vecE_0-k^2 \vecE_0 = 0 $ in $\R^3$, then if this  extended  $\vecE_0$  is
considered as the incident field the corresponding scattered field is
$\vecE^s=0$. In this case, the associated transmission eigenvalues are referred
to as non scattering frequencies. Let us mention that the latter notion is much
more restrictive and it is for instance proven that non scattering frequencies
do not exist in special cases of geometries \cite{BPP2014}. The notion of transmission
eigenvalues is relevant to inverse (spectral) problems as it is shown that
these frequencies can be determined from time-dependent measurements of
scattered waves \cite{cchlsm, LR2015}.

\medskip

The transmission eigenvalue problem is a non-selfadjoint eigenvalue problem
that  is not covered by the standard theory of eigenvalue problems for
elliptic equations. For an introduction we refer to the survey paper
\cite{CaHa2} and the Special Issue of Inverse Problems on  Transmission
Eigenvalues, Volume 29, Number 10, October 2013 \cite{CaHa3}. The discreteness
and existence of real transmission eigenvalues is well understood under the
assumption that the contrast does not change sign in all of $D$
\cite{CaGHa}. Recently, regarding the transmission eigenvalue problem for the
Helmholtz equation, several papers have appeared that address both the question
of discreteness and existence of transmission eigenvalues in the complex plane assuming that the  contrast is  of one sign only in a neighborhood of the  boundary $\partial D$ \cite{CoH, LV1, LV2, R1, S}. The Weyl asymptotic and distribution of transmission eigenvalues are studied in \cite{DP, PV, R2, V}. 

\medskip

The picture is not the same for the transmission eigenvalue problem for the
 Maxwell's equations. The transmission eigenvalues for Maxwell's equations is
 important in application \cite{H}. Some results in this direction are the
 proof of discreteness of transmission eigenvalues in \cite{CaHaMe, Ch} where the magnetic and electric permittivity doesn't change sign
 near the boundary. It is known \cite{CaGHa, CaHa1, CoHa, K}  that, if  $\re(n-I)$ has one sign in $D$ the transmission
 eigenvalues form at most a discrete set without finite accumulation point, and if in addition $\im(n)=0$, there exists an infinite
 set of real transmission eigenvalues. The existence of transmission
 eigenvalues for Maxwell's equations for which the electric permittivity
 changes sign is an open problem. It is our concern to study the existence of
 transmission eigenvalues in the complex plane under the assumption that the electric permittivity
 is constant near the boundary.  Although the index of refraction may be a
 complex valued function, our analysis does not cover the case with absorption
 where the imaginary part of $n$ is proportional to $1/k$. For the case with
 absorption, some non-linear eigenvalue techniques would be more relevant
 \cite{CHL, HiKrOlPa2, Robert}. We also remark that, similarly to the scalar
 case in \cite{R1}, our analysis does not yield information on the existence of
 real transmission eigenvalues. 

\medskip

Now we give an outline of this article with main results. 

\medskip

In section \ref{Formulation} we give an appropriate formulation of the transmission eigenvalue problem and relate transmission eigenvalues to
the eigenvalues of an unbounded linear operator $\vecB_\lambda$. 

\medskip

This motivates us to derive desired regularity results in Section \ref{Regularity} that are needed to show the invertibility of $\vecB_\lambda$ and prove the main theorem. The derivation of these results mainly uses the semi-classical pseudo-differential calculus introduced in \cite{R1} for the scalar case with
appropriate adaptations to Maxwell's system. The assumption  that the electric permittivity
 is constant near the boundary considerably eases the technicality of this
 section and allows us to use results from the scalar problem that are summarized in the Appendix.
{The main technical difficulty related to non constant electric permittivity is
that the divergence free condition is different  for $\vecE$ and $\vecE_0$ near
the boundary. One therefore cannot impose a ``simple'' control of the
divergence of the difference which is needed to establish regularity results.}
\medskip

Using the regularity results obtained in Section \ref{Regularity}, we show that $\vecB_\lambda$ has a bounded
inverse for certain $\lambda$ in Section \ref{InverseBz}. 

\medskip

Section \ref{Main} is dedicated to proving the main
results on transmission eigenvalues following the approach in \cite{R1} which is
based on Agmon's theory for the spectrum of non self-adjoint PDE \cite{A}. We prove for instance
that the inverse $\vecB_\lambda^{-1}$ composed
with a projection operator is a Hilbert-Schmidt operator with desired growth
properties for its resolvent. This allows us to prove that
the set of transmission eigenvalues is discrete, infinite and without finite accumulation
points. Moreover, a notion of generalized
eigenfunctions is introduced and a denseness result is obtained in an
appropriate solution space. The main result is summerized in Theorem \ref{mainHM}.

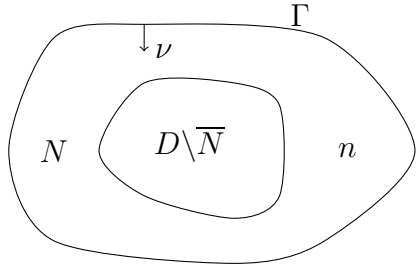
\begin{figure}[hb!] \label{ITEMaxwellFigure1}
        \centering
            \begin{tikzpicture}[scale=0.6]
\draw     plot [smooth cycle] coordinates          { (-3,0)  (-2,-2) (2,-2.5)(4,-2) (6,0) (4,2.5) (0,2.8) (-2,2.5)};
\draw     plot [smooth cycle] coordinates          { (-1,0)  (0,-1) (2,-1.5) (3,-1) (3,1) (2,1.5) (0,1.5) };
              \draw (3,3) node[right]{$\Gamma$};      
              \draw  (-2,-0.5)  node[above, black]{$N$};
              \draw  (1,-0.5) node[above, black]{$D \backslash \overline{N}$};
\draw[->] (0,2.8) -- (0,2.2) node[right]{$\nu$};
\draw (4.5, 0) node[] {$n$};
            \end{tikzpicture} 
\caption{Example of the geometry of the problem}

    \end{figure}
Throughout this article we denote $m:=n-1$ and shall make the following assumption on the index of
refraction $n$.
\begin{assum} \label{Assumptionmconstant}
We assume that the complex valued
  function $n\in C^\infty(\overline D)$ and that $\Re(n) >0 $ in $D$. Moreover
  we assume the existence of a neighborhood $N$ of $\Gamma$ such that $n$ is
  constant in $N$ and that this constant is different from $1$ (which means that
  $m$ is constant and different from zero in $N$). 
\end{assum}
For a complex number $z = |z| e^{i\theta}$,
$\theta \in [0, 2\pi[$ we
define $\arg z := \theta$ and
\begin{eqnarray*}
C(m):=\{ \arg \frac{1}{n(x)} ; \; x \in \overline D\} .
\end{eqnarray*}
For best readability we state the main result of this paper here and refer the readers to Section \ref{Formulation} on the definition of $\vecU(D)$ and $\vecV(D)$, to Section \ref{InverseBz} and Section \ref{Main} on the definition of $\vecS_z$.
\begin{theorem}\label{mainHM}
Assume that Assumption 1 holds and assume that $C(m)$ is contained in an interval of
length $< \frac{\pi}{4}$. Then
there exist infinitely many transmission eigenvalues in the complex plane and
they form a discrete set $\mathcal{T}$ without finite accumulation points. 
 Moreover, there exists $z \in \mathbb{C}$ such that the set
 $\{\mu = (k^2-z)^{-1}, \; k \in \mathcal{T}
 \}$ form the set of eigenvalues of the operator $\vecS_z $ and the associated
 eigenvectors are dense in  $\{\vecu \in \vecU(D); \mdiv \left( \left(1+m\right)\vecu\right)=0\} \times \{\vecv \in \vecV(D); \mdiv  \vecv =0\}$.
\end{theorem}
\section{Formulation of the transmission eigenvalue problem} \label{Formulation}
In the following $D\subset {\mathbb R}^3$ denotes a bounded open and connected
region with $C^\infty$-smooth boundary $\partial D:=\Gamma$ and $\nu$ denotes
the inward unit normal vector on $\Gamma$ (see Figure \ref{ITEMaxwellFigure1} for an example of the geometry). We set $\vecL^2(D):=L^2(D)^3$,
$\vecH^m(D):=H^m(D)^3$ and
define
$$
\vecH(\mcurl\!^2, D):=\left\{\vecu\in \vecL^2(D); \mcurl \vecu\in \vecL^2(D) \; \mbox{and} \; \mcurl \mcurl \vecu \in \vecL^2(D)\right\}
$$
$$
\vecL(\mcurl\!^2, D):=\left\{\vecu\in \vecL^2(D); \mcurl \mcurl \vecu \in \vecL^2(D)\right\}
$$
endowed with the graph norm and define
$$
\vecH_0(\mcurl\!^2, D):=\left\{ \vecu\in \vecH(\mcurl\!^2, D); \gamma_t\vecu=0 \; \mbox{and}\; \gamma_t \mcurl \vecu=0 \; \mbox{on} \; \Gamma \right\}
$$
where $\gamma_t \vecu := \nu \times \vecu|_\Gamma$. 
\begin{definition}
Values of $k\in{\mathbb C}$ for which (\ref{IntroE})-(\ref{IntroNeumann}) has a nontrivial solution $\vecE, \vecE_0 \in \vecL(\mcurl\!^2,D)$ and  $\vecE-\vecE_0 \in \vecH_0(\mcurl\!^2, D)$ are called transmission eigenvalues.
\end{definition}

Following the approach in \cite{R1, S} for the scalar case, we rewrite the
transmission eigenvalue problem in an equivalent form in terms of   $\vecu := \vecE-\vecE_0 \in
\vecH_0(\mcurl\!^2, D)$ and 
$\vecv:=k^2\vecE_0 \in \vecL(\mcurl\!^2,D)$ 
\begin{eqnarray}
&\mcurl \mcurl \vecu-k^2 (1+m) \vecu - m \vecv= 0 \quad &\mbox{in} \quad D  \label{IntroEu} \\
&\mcurl \mcurl \vecv-k^2 \vecv = 0 \quad &\mbox{in} \quad D \label{IntroEv}
\end{eqnarray}

\begin{definition}
Normalized non-trivial solutions $\vecu \in \vecH_0(\mcurl\!^2, D)$ and $\vecv \in \vecL(\mcurl\!^2,D)$ to equations (\ref{IntroEu})-(\ref{IntroEv}) are called transmission eigenvectors corresponding to $k$.
\end{definition}
\subsection{Function spaces for the transmission eigenvectors}
To study the PDEs (\ref{IntroEu})-(\ref{IntroEv}) and formulate the transmission eigenvalue problem, we first investigate the function spaces that transmission eigenvectors $\vecu$ and $\vecv$ belong to. This is the motivation of the next lemma.

\begin{lemma} \label{FormulationFunctionalSpace}
Assume that assumption \ref{Assumptionmconstant} holds and $\vecu \in \vecH_0(\mcurl\!^2, D)$ and $\vecv \in \vecL(\mcurl\!^2,D)$ are transmission eigenvectors corresponding to $k$. Then $\mdiv \vecu \in H^1(D)$, $\mdiv \vecv \in H^1(D)$ and in particular $\mdiv \vecv =0$.
\end{lemma}
\begin{proof}
Taking the divergence of \eqref{IntroEv} implies $\mdiv \vecv =0$ and therefore $\mdiv \vecv \in H^1(D)$. Taking the divergence of equation (\ref{IntroEu}) yields
\begin{eqnarray}
\label{divform}
(1+m)\mdiv \vecu + \nabla m \cdot \vecu = -k^{-2}(\nabla  m \cdot \vecv +m \mdiv \vecv).
\end{eqnarray}
Since  $\nabla m$ has compact support in $D$ and $\vecv$ satisfies a vectorial Helmholtz equation in $D$,  then standard
regularity results give $\nabla m \cdot \vecv \in H^1(D)$. Since $\mdiv \vecv \in H^1(D)$ and $\vecu \in \vecL^2(D)$, we deduce
from \eqref{divform} that $\mdiv \vecu \in L^2(D)$. Since $\mcurl \vecu \in \vecL^2(D)$ and $\gamma_t \vecu =0$, $\vecu \in \vecH^1(D)$ (c.f. \cite{ABDG}).  Hence, using again \eqref{divform}, $\mdiv \vecu \in
H^1(D)$ and we have proved the lemma.
\end{proof} \proofend

We now define the following spaces:
$$
\vecU(D) := \left\{\vecu\in \vecH_0(\mcurl\!^2, D); \mdiv \vecu \in H^1(D) \right\}
$$
and
$$
\vecV(D):=\left\{\vecv\in \vecL^2(D); \mcurl \mcurl \vecv \in \vecL^2(D) \; \mbox{and} \; \mdiv \vecv \in H^1(D) \right\}.
$$
\subsection{Relating the transmission eigenvalues to the spectrum of an operator}
Having studied the function spaces that transmission eigenvectors belong to, we are ready to introduce an operator which plays an important role in our analysis. We introduce the operator
$\vecB_\lambda$ defined on $\vecU(D)  \times \vecV(D)$ by
$$
\vecB_\lambda(\vecu, \vecv) = (\vecf,\vecg)
$$
where
\begin{eqnarray}
&\mcurl \mcurl \vecu-\lambda (1+m) \vecu - m \vecv= (1+m)\vecf \quad &\mbox{in} \quad D  \label{FormulationEu} \\
&\mcurl \mcurl \vecv-\lambda \vecv = \vecg \quad &\mbox{in} \quad D \label{FormulationEv}
\end{eqnarray}
and $\lambda \in \C$ is a fixed parameter (we will choose $\lambda$ later). We can now relate the transmission eigenvalue with the eigenvalues of $\vecB_\lambda$. In fact,
one observes that $k$ is a transmission eigenvalue if and only if $k^2-\lambda$
is an eigenvalue of $\vecB_\lambda$ (this also explains the motivation to define the operator $\vecB_\lambda$). 

\medskip

To study the invertibility of the operator $\vecB_\lambda$, we first investigate the range of $\vecB_\lambda$.
\begin{lemma}
Assume $\vecB_\lambda(\vecu, \vecv) = (\vecf,\vecg)$ and $(\vecu,\vecv) \in \vecU(D)  \times \vecV(D)$. Then $\vecf \in \vecL^2(D)$, $\mdiv \left( \left( 1+m \right)\vecf \right) \in H^1(D)$, $\vecg \in \vecL^2(D)$ and $\mdiv \vecg \in H^1(D)$.
\end{lemma}
\begin{proof}
Noting that $\vecv \in \vecV$ and $\mcurl \!^2 = \nabla \mdiv -\Delta$, we have that
$$
\Delta \vecv = \nabla \mdiv \vecv -\mcurl \!^2 \vecv \in \vecL^2(D).
$$
Since $\nabla m$ has compact support in $D$, standard elliptic regularity results yield $\nabla m \cdot \vecv \in H^2(D)$. Since
$$
\mdiv \left( m \vecv \right) = \nabla m \cdot \vecv + m \mdiv v,
$$
we have that
$$
\mdiv \left( m \vecv \right) \in H^1(D).
$$
Since $\vecu \in \vecU$, $\vecu \in \vecH^2(D)$(c.f. \cite{ABDG}). Therefore
$$
\mdiv \left( \left( 1+m \right)\vecf \right) = -\lambda \mdiv \left( \left( 1+m \right)\vecu \right) - \mdiv \left( m \vecv \right) \in H^1(D).
$$
$\mdiv \vecg \in H^1(D)$ follows directly from $\mdiv \vecv \in H^1(D)$. This proves our lemma.
\end{proof} \proofend

\medskip

We now define the following spaces:
$$
\vecF(D) := \left\{\vecf\in \vecL^2(D); \mdiv \left( \left( 1+m \right)\vecf \right)  \in H^1(D) \right\}
$$
and
$$
\vecG(D):=\left\{\vecg \in \vecL^2(D); \mdiv \vecg \in H^1(D)\right\}.
$$
\section{Regularity results for transmission eigenvectors} \label{Regularity}

As is seen from Section \ref{Formulation}, the analysis of transmission eigenvalues
will be obtained from the analysis of the spectrum of the operator
$\vecB_\lambda$ or more precisely of its inverse $\vecR_\lambda$. To show the existence of $\vecR_\lambda$ for well chosen $\lambda$, we need certain regularity results and this is the purpose of this section. Moreover, the regularity results in this section (in particular Theorem \ref{Regularity2}) is important to apply the spectral theory of Hilbert-Schmidt operator in section \ref{Main}. The reader may proceed to read section \ref{InverseBz} and section \ref{Main} by assuming Theorem \ref{Regularity1} and \ref{Regularity2} and come back to the technical details in this section after that.
\medskip

In this section we will derive a detailed study of equations
(\ref{FormulationEu})-(\ref{FormulationEv}). Roughly speaking we will show that, for appropriate $\lambda$ the solutions $\vecu$ and $\vecv$ are bounded by $\vecf$ and $\vecg$ in appropriate norms. The idea is based on applying the semiclassical pseudo-differential calculus
used in \cite{R1} for the scalar problem. The analysis for Maxwell's equations
requires non trivial adaptations since the normal component of the
trace of $\vecu$ does not necessarily vanish, the $\mcurl \mcurl$ operator is
not strongly elliptic and the compact embedding for Maxwell's equations are
more complicated. Restricting ourselves to the case $m$ is constant
near the boundary simplifies the analysis since one can first derive a semiclassical estimate for the normal component of
the trace of $\vecv$. This allows us to then derive estimates for $\vecu$ and
$\vecv$.  In order to write the equation for the normal trace of $\vecv$ and
apply the analysis in \cite{R1} we first need to rewrite
(\ref{FormulationEu})-(\ref{FormulationEv}) as a problem in $\R^3$.

\subsection{Extending solutions to $\R^3$}
To begin with, we introduce a tubular neighborhood $D_\epsilon$ of $\Gamma$, where 
$$
D_\epsilon=\left\{ x: x= y+ s \nu(y), \, y \in \Gamma, \, 0 \le s < \epsilon \right\}.
$$ 

We define
$$
\Gamma_s=\left\{ x: x= y+ s \nu(y), \, y \in \Gamma \right\}.
$$ 
The boundary $\Gamma$ corresponds to $\Gamma_s$ with $s=0$.

\medskip

To deal with the boundary conditions on $\Gamma$, we follow the idea in \cite{R1} and extend the transmission eigenvectors by $0$ outside $D$.  To begin with, let us introduce
\begin{equation*}
\uvecu=\left\{ \begin{array}{cc}
\vecu(x) &\qquad \mbox{in} \quad D\\
0 &\qquad \mbox{in} \quad {\R}^3 \backslash \overline{D}.
\end{array}
\right.
\end{equation*}
\begin{lemma}
Assume $(\vecf,\vecg)=\vecB_\lambda (\vecu,\vecv)$ as defined by equations (\ref{FormulationEu}) and (\ref{FormulationEv}). Then $\uvecu$ and $\uvecv$ satisfy the following
\begin{eqnarray} 
-\Delta \uvecu -\lambda (1+m) \uvecu - m \uvecv = (1+m)\underline{\vecf} - \nabla \underline{\mdiv \vecu} -  \nabla_\Gamma (\vecu_N \cdot \nu)  \otimes \delta_{s=0} - \vecu_N \otimes D_s \delta_{s=0} \label{TubularNeighborhoodEquationu}
\end{eqnarray}
\begin{eqnarray} 
-\Delta \uvecv -\lambda \uvecv = \underline{\vecg} + \lambda^{-1}\nabla \underline{ \mdiv \vecg} - (2H\vecv_{T}+ \frac{\partial \vecv_{T}}{\partial \nu}- \nu \mdiv\!_\Gamma \vecv_T) \otimes \delta_{s=0} - \vecv \otimes D_s \delta_{s=0}  \label{TubularNeighborhoodEquationv}
\end{eqnarray}
where $\gamma \vecu := \vecu |_\Gamma$, $\vecu_T :=\gamma_T \vecu:=\nu \times (\vecu \times \nu)|_\Gamma$ and $\vecu_N :=\gamma_N \vecu=\nu (\vecu \cdot \nu)|_\Gamma$. Here $\delta_{s=0}$ is the delta distribution on $\Gamma$ and $D_s$ is the normal derivative.
\end{lemma}
\begin{proof}
From $\Delta$ in geodesic coordinates (c.f. \cite{R1}  and \cite{N}),
we have that
\begin{eqnarray*} 
&&\Delta \uvecu = \underline{\Delta \vecu} + (2H \vecu_T+ \frac{\partial \vecu_T}{\partial \nu}+2H\vecu_N + \frac{\partial \vecu_N}{\partial \nu}) \otimes \delta_{s=0} + (\vecu_T + \vecu_N)\otimes D_s \delta_{s=0}
\end{eqnarray*}
where $H$ is a smooth function on $\Gamma$ (see Appendix).
From $\mcurl \!^2 = \nabla \mdiv - \Delta$ we are able to rewrite the equations (\ref{FormulationEu})-(\ref{FormulationEv}) as follows 
\begin{eqnarray}\label{LemmaTubularu}
-\Delta \uvecu -\lambda (1+m) \uvecu - m \uvecv &=& (1+m)\underline{\vecf} - \underline{\nabla \mdiv \vecu} - (\vecu_{T} + \vecu_N)\otimes D_s \delta_{s=0}  \nonumber \\
&-&  (2H\vecu_{T}+ \frac{\partial \vecu_{T}}{\partial \nu}+2H\vecu_N + \frac{\partial \vecu_N}{\partial \nu}) \otimes \delta_{s=0} \label{LemmaTubularv}
\end{eqnarray}
and
\begin{eqnarray} \label{LemmaTubularv}
-\Delta \uvecv -\lambda \uvecv &=& \underline{\vecg} -\underline{\nabla \mdiv \vecv} - (2H\vecv_{T}+ \frac{\partial \vecv_{T}}{\partial \nu}+2H\vecv_N + \frac{\partial \vecv_N}{\partial \nu}) \otimes \delta_{s=0} \nonumber \\
&-& (\vecv_{T} + \vecv_N)\otimes D_s \delta_{s=0} .
\end{eqnarray}
We now use the fact that (c.f. \cite{N})
\begin{eqnarray*}
\nabla \underline{\mdiv \vecu} &=& \underline{\nabla \mdiv \vecu} + (\nu \mdiv \vecu) \otimes \delta_{s=0} \\
\nu \mdiv \vecu &=& \nu \mdiv\!_{\Gamma} \vecu_T + 2H \vecu_N + \frac{\partial \vecu_N}{\partial \nu}  \quad \mbox{on} \quad \Gamma
\end{eqnarray*}
with the same equations hold for $\vecv$. Using above two equations to simplify equations (\ref{LemmaTubularu})-(\ref{LemmaTubularv}) we get
\begin{eqnarray*}
-\Delta \uvecu -\lambda (1+m) \uvecu - m \uvecv &=& (1+m)\underline{\vecf} - \nabla \underline{ \mdiv \vecu} - (\vecu_{T} + \vecu_N)\otimes D_s \delta_{s=0} \\
&-&  (2H\vecu_{T}+ \frac{\partial \vecu_{T}}{\partial \nu}- \nu \mdiv\!_{\Gamma} \vecu_T ) \otimes \delta_{s=0}  
\end{eqnarray*}
and
\begin{eqnarray*}
-\Delta \uvecv -\lambda \uvecv &=& \underline{\vecg} -\nabla \underline{ \mdiv \vecv}- (2H\vecv_{T}+ \frac{\partial \vecv_{T}}{\partial \nu}-\nu \mdiv\!_{\Gamma} \vecv_T ) \otimes \delta_{s=0} \\
&-& (\vecv_{T} + \vecv_N)\otimes D_s \delta_{s=0} .
\end{eqnarray*}
We now use (c.f. \cite{N})
$$
\nu \times \mcurl \vecu = \nabla_\Gamma (\vecu_N \cdot \nu) + \nu \times (R (\vecu \times \nu)) -2H\vecu_T- \frac{\partial \vecu_{T}}{\partial \nu}  \quad \mbox{on} \quad \Gamma .
$$
Then $\vecu_T=0$ and $(\mcurl \vecu) _T=0$ yields
$$
\frac{\partial \vecu_{T}}{\partial \nu}  = \nabla_\Gamma (\vecu_N \cdot \nu) \quad \mbox{on} \quad \Gamma 
$$
and therefore we get (\ref{TubularNeighborhoodEquationu}). From equation (\ref{FormulationEv})
\begin{eqnarray*}
-\lambda \mdiv \vecv = \mdiv \vecg.
\end{eqnarray*}
This yields equation (\ref{TubularNeighborhoodEquationv}).
\end{proof} \proofend

The following lemma is important in our analysis as it allows us in subsection \ref{SectionRegularity1} to derive an estimate only involving $\vecv_N$.
\begin{lemma}
Assume $(\vecf,\vecg)=\vecB_\lambda (\vecu,\vecv)$ as defined by equations (\ref{FormulationEu}) and (\ref{FormulationEv}). Then 
\begin{eqnarray*}
\lambda \vecu_N=- \frac{m}{(1+m)}\vecv_N - \vecf_N .
\end{eqnarray*}
In particular for $\lambda=h^{-2} \mu$ where $h>0$ and $\mu \not=0 \in \C$, we have
\begin{eqnarray} \label{VectorialPotentialsutov}
\vecu_N=-h^2 \frac{m}{\mu(1+m)}\vecv_N - h^2 \frac{1}{\mu} \vecf_N .
\end{eqnarray}
\end{lemma}
\begin{proof}
Equation (\ref{FormulationEu}) yields
\begin{eqnarray*}
\lambda (1+m) \vecu_N=- m \vecv_N - (1+m)\vecf_N + \mcurl\mcurl \vecu \cdot \nu.
\end{eqnarray*}
Since $\mcurl \vecu \times \nu=0$, then $\mcurl\mcurl \vecu \cdot \nu=-\mdiv\!_\Gamma (\mcurl \vecu \times \nu)=0$. Then we can prove the lemma.
\end{proof} \proofend
\subsection{A first regularity result} \label{SectionRegularity1}
We prove in this subsection a first explicit continuity result for $(\vecu, \vecv) \in \vecU(D) \times \vecV(D)$
satisfying 
$$\vecB_\lambda(\vecu,\vecv)=(\vecf,\vecg)$$
 for certain large values of $\lambda$. We refer to the Appendix for notations related to pseudo-differential
calculus and some key results from \cite{R1}. Readers may need to read the Appendix first to be able to understand the proof. 

\medskip

Throughout this section, we let $h:=\frac{1}{|\lambda|^{\frac{1}{2}}}$ and $\mu:=h^2 \lambda$. Multiplying equations
(\ref{TubularNeighborhoodEquationu}) and  (\ref{TubularNeighborhoodEquationv})  by $h^2$ yields
\begin{eqnarray} \label{1stRegularityPDE1}
-h^2\Delta \uvecu -\mu (1+m) \uvecu - h^2 m \uvecv &=& h^2(1+m)\underline{\vecf} + \frac{h}{i} \nabla_h \underline{\mdiv \vecu}  \nonumber \\
&+&  \frac{h}{i} \nabla_\Gamma^h(\vecu_N \cdot \nu) \otimes \delta_{s=0} +  \frac{h}{i}\vecu_N \otimes D^h_s \delta_{s=0} 
\end{eqnarray}
and
\begin{eqnarray} \label{1stRegularityPDE2}
-h^2\Delta \uvecv -\mu \uvecv &=& h^2\underline{\vecg} -\frac{h^3}{i\mu} \nabla_h \underline{\mdiv \vecg} \nonumber \\
&+& \frac{h}{i}( 2 \frac{h}{i} H\vecv_{T}+ \frac{\partial_h \vecv_{T}}{\partial \nu}- \nu \mdiv\!_\Gamma^h \vecv_T ) \otimes \delta_{s=0} + \frac{h}{i} \vecv \otimes D^h_s \delta_{s=0}
\end{eqnarray}
(see Appendix for notations of $D_{x_j}^h$, $\nabla_h$, $\frac{\partial_h}{\partial \nu}$).
We define $\vecJ (\vecv_T)$ by
$$
\vecJ(\vecv_T):= 2 \frac{h}{i} H\vecv_{T}+ \frac{\partial_h \vecv_{T}}{\partial \nu}- \nu \mdiv\!_\Gamma^h \vecv_T.
$$
Based on these two equations, we will derive the desired regularity results. 

\medskip

Before digging into the technical estimates, we first explain the ideas and what we are doing in each Lemma and Theorem. The general idea is to get first an estimate for
$\vecv_N$ and $\vecu_N$. This will allow us to derive estimates for $\vecv$ and
$\vecJ(\vecv_T)$ and consequently estimates for $\vecv$ and $\vecu$. 

\medskip

More specifically, it will be seen in Theorem \ref{Regularity1} that the
estimates of $\vecu$ and $\vecv$ stems from the estimates of
$\vecv_N$ in ${\vecH^{-\frac{1}{2}}_{sc}(\Gamma)}$ and of 
$\vecJ(\vecv_T)$ in ${\vecH^{-\frac{3}{2}}_{sc}(\Gamma)}$ evidenced from
(\ref{SectionRegularity1Step4VJ}) and (\ref{Section2Step4Estimatev}). To get an 
estimate for $\vecv_N$ in ${\vecH^{-\frac{1}{2}}_{sc}(\Gamma)}$ we will need to 
get an estimate for $\vecg_5$ in ${\vecH^{\frac{3}{2}}_{sc}(\Gamma)}$ as is seen from 
(\ref{sec2Regularity1Estimatevbyg5}). The estimate for $\vecg_5$ is obtained by
establishing an equation for 
$\vecu_N$ that allows us to control the ${\vecH^{\frac{3}{2}}_{sc}(\Gamma)}$
norm of
this boundary term. This is the first main additional
technical  difference between the scalar problem treated in \cite{R1}
and the present one. For the scalar case this step in not needed since the 
solution has vanishing traces on the boundary. 

\medskip

Therefore, Lemma \ref{1stRegularityuLemma}, Lemma \ref{1stRegularityvNLemma},
Lemma \ref{1stRegularityvNuLemma} and Lemma \ref{1stRegularityuNLemma}  serve
to derive the desired estimate for $\vecu_N$ in
${\vecH^{\frac{3}{2}}_{sc}(\Gamma)}$. In Lemma \ref{1stRegularityuNLemma}, we
derive an estimate for $\vecu_N$ that only involves $\vecv$, $\vecf$ and
$\vecg$. This will serve to obtain an estimate for $\vecv$ in Theorem
\ref{Regularity1}. The estimate of $\vecu_N$
in ${\vecH^{\frac{3}{2}}_{sc}(\Gamma)}$ stems from estimate of $\vecv_N$ in
${\vecH^{-\frac{1}{2}}_{sc}(\Gamma)}$. This is the motivation of Lemma
\ref{1stRegularityvNuLemma}: an a priori estimate on $\vecv_N$ independent of
$\vecu$. To fullfill this, we derive an a priori estimate for $\vecv_N$ (involving $\vecu$) in Lemma \ref{1stRegularityvNLemma} and an a priori estimate on $\vecu$ involving $\vecv_N$ in Lemma \ref{1stRegularityuLemma} (such that we can eliminate $\vecu$ in Lemma \ref{1stRegularityvNuLemma}). 

\medskip

Now we begin with the following lemma.

\begin{lemma} \label{1stRegularityuLemma}
Assume that assumption \ref{Assumptionmconstant} holds. Assume in addition that $|\xi|^2-\mu \not = 0$, $|\xi|^2- (1+m)\mu \not = 0$ for any $\xi$ and $x \in \overline{D}$.  Then for sufficiently small $h$
\begin{eqnarray} \label{Regularityu}
\|\vecu\|_{\vecL^2(D)} &\lesssim& h^2 \|\vecv\|_{\vecL^2(D)}+h^2 \|\vecf\|_{\vecL^2(D)}+h^5 \|\vecg\|_{\vecL^2(D)}+ h^5 \|\mdiv \vecg\|_{L^2(D)}  \nonumber \\
&+&h^2 \|\mdiv \vecf\|_{L^2(D)} +  h^{\frac{5}{2}}|\vecv_N|_{\vecH^{-\frac{1}{2}}_{sc}(\Gamma)} .
\end{eqnarray}
\end{lemma}
\begin{proof}
From the Appendix, $Q$ is a parametrix of $-h^2\Delta -\mu (1+m)$, then applying $Q$ to equation (\ref{1stRegularityPDE1})
\begin{eqnarray}
\uvecu &=& h K_{-M} \uvecu + h^2 Q(m \uvecv) + h^2Q((1+m)\underline{\vecf}) + \frac{h}{i}Q( \nabla_h \underline{\mdiv \vecu}) \nonumber \\
&+&  Q(\frac{h}{i}\nabla_\Gamma^h(\vecu_N \cdot \nu) \otimes \delta_{s=0}) +  Q(\frac{h}{i}\vecu_N \otimes D^h_s \delta_{s=0})  \label{Section2Step1Eqnu}
\end{eqnarray}
where $K_{-M}$ denotes a semiclassical pseudo-differential operator of order $-M$ with $M$ positive and sufficiently large. 
From equation (\ref{Section2Step1Eqnu}), estimate (\ref{PreliminaryResultsBounaryToDomain}) and Lemma \ref{PreliminaryResultsQNablau}
\begin{eqnarray} \label{Section2Step2Estimateu}
\|\vecu\|_{\vecL^2(D)} &\lesssim& h^2 \|\vecv\|_{\vecL^2(D)}+h^2 \|\vecf\|_{L^2(D)}+h \|\mdiv \vecu\|_{\vecL^2(D)} +  h^{\frac{1}{2}}|\vecu_N|_{\vecH^{-\frac{1}{2}}_{sc}(\Gamma)} \nonumber \\
&+& h^{\frac{1}{2}}|\nabla_\Gamma^h(\vecu_N \cdot \nu)|_{\vecH^{-\frac{3}{2}}_{sc}(\Gamma)} .
\end{eqnarray}
Then a direct calculation (see the Calculation subsection \ref{RegularityCalculation}) yields the lemma.
\end{proof} \proofend

With reference to Appendix \ref{PreliminaryResultsPrincipleSemiclassicalSymbol} on the definition of $R_0(x,\xi')$, we begin with the following lemma.
\begin{lemma} \label{1stRegularityvNLemma}
Assume that assumption \ref{Assumptionmconstant} holds. Assume in addition that $|\xi|^2-\mu \not = 0$ , $|\xi|^2- (1+m)\mu \not = 0$ for any $\xi$ and $x \in \overline{D}$ and $R_0(x,\xi')-\frac{1+m}{2+m}\mu \not =0 $ for any $\xi'$ and $x \in \Gamma$.  Then for sufficiently small $h$
\begin{eqnarray} \label{RegularityvNormal}
|\vecv_N|_{\vecH^{-\frac{1}{2}}_{sc}(\Gamma)} &\lesssim& h^{\frac{3}{2}} \|\vecg\|_{\vecL^2(D)} + h^{\frac{5}{2}} \|\mdiv \vecg\|_{\vecL^2(D)} + h^{-\frac{1}{2}} \|\vecf\|_{\vecL^2(D)} + h^{-\frac{1}{2}} \|\mdiv \vecf\|_{L^2(D)} \nonumber \\
&+& h^{\frac{1}{2}}\|\vecv\|_{\vecL^2(D)} + h^{-\frac{3}{2}}\|\vecu\|_{\vecL^2(D)} + h^{-\frac{3}{2}}\|\mdiv \vecu\|_{L^2(D)} \nonumber \\
&+& h |\vecJ(\vecv_T)|_{\vecH^{-\frac{3}{2}}_{sc}(\Gamma)} + h |\gamma \vecv|_{\vecH^{-\frac{1}{2}}_{sc}(\Gamma)} .
\end{eqnarray}
\end{lemma}
\begin{proof}
The idea is to derive an equation for $\vecv_N$, which we will do in Steps 1, 2, and 3. In Step 4, we then derive an a priori estimate for $\vecv_N$.

\medskip

\textit{Step 1}: Relating $\vecv_N$ to $\mdiv\!^h_\Gamma \vecv_\Gamma$.

\medskip

From the Appendix, $\tilde{Q}$ is a parametrix of $-h^2\Delta -\mu$. Then applying $\tilde{Q}$ to equation (\ref{1stRegularityPDE2}) we have that
\begin{eqnarray} \label{Section2SubStep1Eqnv}
\uvecv &=& h K_{-M} \uvecv + h^2 \tilde{Q} \underline{\vecg}  -  h^3\tilde{Q} (\frac{1}{i \mu} \nabla_h \underline{\mdiv \vecg})  \nonumber \\
&+& \tilde{Q}[\frac{h}{i} \vecJ(\vecv_T) \otimes \delta_{s=0} ] + \tilde{Q} [\frac{h}{i}\vecv \otimes D^h_s \delta_{s=0}] .
\end{eqnarray}
Taking the traces on the boundary $\Gamma$ and a direct calculation (see the Calculation subsection \ref{RegularityCalculation}) yields
\begin{eqnarray}
- \nu \mdiv\!^h_\Gamma \vecv_T + \mop(\rho_2)\vecv_N&=& \mop(r_{1})\left( h \gamma_N K_{-M} \uvecv + h^2 \gamma_N \tilde{Q} \underline{\vecg} -  h^3\gamma_N\tilde{Q} (\frac{1}{i \mu} \nabla_h \underline{\mdiv \vecg})  \right) \nonumber\\
&+& h\mop(r_{-1}) \vecJ(\vecv_T) + h\mop(r_{0}) \vecv \nonumber \\
&+& h\mop(r_{-1}) (-\nu \mdiv\!^h_\Gamma \vecv_T) + h\mop(r_{0}) \vecv_N \nonumber \\
&:=& \vecg_1 \label{Section2Substep1Eqnv}
\end{eqnarray}
where we denote the right hand side as $\vecg_1$. 

\medskip

{\it Step 2}. Relating $\vecu_N$ to $\vecv_N$. 

\medskip

Using a similar argument as in Step 1 (see the Calculation subsection \ref{RegularityCalculation}) yields
\begin{eqnarray}
\vecu_N&=& h \gamma_N K_{-M} \uvecu + h^3 \gamma_N Q(m K_{-M}\uvecv) + h^4 \gamma_N Qm\tilde{Q} \underline{\vecg} - h^5\gamma_N Qm\tilde{Q} (\frac{1}{i \mu} \nabla_h \underline{\mdiv \vecg})  \nonumber \\
&+& h^2 \gamma_N Q((1+m)\underline{\vecf}) + \frac{h}{i}\gamma_N Q( \nabla_h \underline{\mdiv \vecu}) \nonumber \\
&+& h^2 \mop \left( \frac{m(\rho_2-\rho_1+\lambda_2-\lambda_1)}{(\lambda_1-\lambda_2)(\lambda_1-\rho_2)(\rho_1-\lambda_2)(\rho_1-\rho_2)} \right) (-\nu \mdiv\!^h_\Gamma \vecv_T) \nonumber \\
 &+&  h^2 \mop \left( \frac{m(\rho_2\lambda_2-\rho_1\lambda_1)}{(\lambda_1-\lambda_2)(\lambda_1-\rho_2)(\rho_1-\lambda_2)(\rho_1-\rho_2)}\right) \vecv_N + \mop(\frac{\lambda_1}{\lambda_1-\lambda_2} ) \vecu_N \nonumber \\
&+& h^3 \mop(r_{-4}) \vecJ(\vecv_T) + h^3 \mop(r_{-3})\vecv+h \mop(r_{-1})\vecu_N + h \mop(r_{-2})\nabla_\Gamma^h(\vecu_N \cdot \nu) \nonumber \\
&:=& h^2 \mop \left( \frac{m(\rho_2-\rho_1+\lambda_2-\lambda_1)}{(\lambda_1-\lambda_2)(\lambda_1-\rho_2)(\rho_1-\lambda_2)(\rho_1-\rho_2)} \right) (-\nu \mdiv\!^h_\Gamma \vecv_T)  \nonumber \\
 &+&  h^2 \mop \left( \frac{m(\rho_2\lambda_2-\rho_1\lambda_1)}{(\lambda_1-\lambda_2)(\lambda_1-\rho_2)(\rho_1-\lambda_2)(\rho_1-\rho_2)} \right)\vecv_N + \mop(\frac{\lambda_1}{\lambda_1-\lambda_2})\vecu_N + \vecg_2 . \label{Section2Substep2Eqnu}
\end{eqnarray}

{\it Step 3}. Derive an equation for $\vecv_N$. 

\medskip

From equation (\ref{VectorialPotentialsutov}) $\vecu_N=-h^2 \frac{m}{\mu(1+m)}\vecv_N - h^2 \frac{1}{\mu} \vecf_N$. Then, combining this with equations (\ref{Section2Substep1Eqnv}) and (\ref{Section2Substep2Eqnu}) yields
\begin{eqnarray*}
&& -h^2 \frac{m}{\mu(1+m)}\vecv_N - h^2 \frac{1}{\mu} \vecf_N \\
=&& h^2 \mop(\frac{m(\rho_2\lambda_2-\rho_1\lambda_1)}{(\lambda_1-\lambda_2)(\lambda_1-\rho_2)(\rho_1-\lambda_2)(\rho_1-\rho_2)})\vecv_N + \mop(\frac{\lambda_1}{\lambda_1-\lambda_2})(-h^2 \frac{m}{\mu(1+m)}\vecv_N - h^2 \frac{1}{\mu} \vecf_N) \\
+&& h^2 \mop(\frac{m(\rho_2-\rho_1+\lambda_2-\lambda_1)}{(\lambda_1-\lambda_2)(\lambda_1-\rho_2)(\rho_1-\lambda_2)(\rho_1-\rho_2)}) (-\mop(\rho_2)\vecv_N+ \vecg_1) +\vecg_2 .
\end{eqnarray*}
Hence
\begin{eqnarray*}
&& h^2 \mop\left(  -\frac{m}{\mu(1+m)} -\frac{m(\rho_2\lambda_2-\rho_1\lambda_1)-m\rho_2(\rho_2-\rho_1+\lambda_2-\lambda_1)}{(\lambda_1-\lambda_2)(\lambda_1-\rho_2)(\rho_1-\lambda_2)(\rho_1-\rho_2)} + \frac{m}{\mu(1+m)} \frac{\lambda_1}{\lambda_1-\lambda_2}  \right) \vecv_N\\
=&& h^2 \mop(r_{0}) \vecf_N + \vecg_2+ h^2\mop(r_{-3})\vecg_1 := \vecg_3 .
\end{eqnarray*}

{\it Step 4}. Getting an a priori estimate for $\vecv_N$.

\medskip

From equations (\ref{PreliminaryResultsPolynomial1}) and (\ref{PreliminaryResultsPolynomial2}) we have $\lambda_1=-\lambda_2$, $\rho_1=-\rho_2$, $-\lambda_2^2=R-\mu(1+m)$ and $-\rho_2^2=R-\mu$. Then a direct calculation yields
\begin{eqnarray*}
&&-\frac{m}{\mu(1+m)} -\frac{m(\rho_2\lambda_2-\rho_1\lambda_1)-m\rho_2(\rho_2-\rho_1+\lambda_2-\lambda_1)}{(\lambda_1-\lambda_2)(\lambda_1-\rho_2)(\rho_1-\lambda_2)(\rho_1-\rho_2)} + \frac{m}{\mu(1+m)} \frac{\lambda_1}{\lambda_1-\lambda_2} \\
=&& \frac{1}{2(1+m)\mu} \frac{\lambda_2-(1+m)\rho_2}{\lambda_2}.
\end{eqnarray*}
Then 
\begin{eqnarray*}
\mop(\lambda^2_2-(1+m)^2\rho^2_2) \vecv_N = h^{-2} \mop(2(1+m)\mu \lambda_2(\lambda_2+(1+m)\rho_2))\vecg_3 + h\mop(r_{1})\vecv_N,
\end{eqnarray*}
which implies that
\begin{eqnarray*}
\mop\left( m((m+2)R-(1+m)\mu) \right) \vecv_N = h^{-2} \mop(r_2)\vecg_3 + h\mop(r_{1})\vecv_N .
\end{eqnarray*}
Let $R_0(x,\xi')$ be the principal symbol of $R(x,\xi')$, see also Appendix \ref{PreliminaryResultsPrincipleSemiclassicalSymbol}. Then
\begin{eqnarray*}
\mop\left( m((m+2)R_0-(1+m)\mu) \right) \vecv_N = h^{-2} \mop(r_2)\vecg_3 + h\mop(r_{1})\vecv_N .
\end{eqnarray*}
Note that
\begin{eqnarray} \label{RegularityAssumptionBoundary}
(m+2)R_0-(1+m)\mu \not =0
\end{eqnarray}
for any $\xi'$ and $x \in \Gamma$. Then there exists a parametrix of $(m+2)R_0-(1+m)\mu$ and consequently
\begin{eqnarray*}
&&|\vecv_N|_{\vecH^{-\frac{1}{2}}_{sc}(\Gamma)} \\
&\lesssim& h^{-2}|\vecg_3|_{\vecH^{-\frac{1}{2}}_{sc}(\Gamma)} + h|\vecv_N|_{\vecH^{-\frac{3}{2}}_{sc}(\Gamma)} \\
&\lesssim&  |\vecf_N|_{\vecH^{-\frac{1}{2}}_{sc}(\Gamma)} + h^{-2}|\vecg_2|_{\vecH^{-\frac{1}{2}}_{sc}(\Gamma)}+ |\vecg_1|_{\vecH^{-\frac{1}{2}}_{sc}(\Gamma)} +h|\vecv_N|_{\vecH^{-\frac{3}{2}}_{sc}(\Gamma)}
\end{eqnarray*}
A direct calculation (see the Calculation subsection \ref{RegularityCalculation}) yields the lemma.
\end{proof} \proofend

Now Lemma \ref{1stRegularityuLemma} and Lemma \ref{1stRegularityvNLemma} now yield the following.
\begin{lemma}  \label{1stRegularityvNuLemma}
Assume that assumption \ref{Assumptionmconstant} holds. Assume in addition that $|\xi|^2-\mu \not = 0$ , $|\xi|^2- (1+m)\mu \not = 0$ for any $\xi$ and $x \in \overline{D}$, and $R_0(x,\xi')-\frac{1+m}{2+m}\mu \not =0 $ for any $\xi'$ and $x \in \Gamma$.  Then for sufficiently small $h$
\begin{eqnarray} \label{Regularity1Step3vN}
|\vecv_N|_{\vecH^{-\frac{1}{2}}_{sc}(\Gamma)} &\lesssim & h^{\frac{3}{2}} \|\vecg\|_{\vecL^2(D)} + h^{\frac{5}{2}}  \|\mdiv \vecg\|_{L^2(D)} + h^{-\frac{1}{2}} \|\vecf\|_{\vecL^2(D)} \nonumber \\
&+& h^{\frac{1}{2}}\|\vecv\|_{\vecL^2(D)} + h^{-\frac{1}{2}}\|\mdiv \vecf\|_{L^2(D)} + h |\vecJ(\vecv_T)|_{\vecH^{-\frac{3}{2}}_{sc}(\Gamma)} + h |\gamma \vecv|_{\vecH^{-\frac{1}{2}}_{sc}(\Gamma)} .
\end{eqnarray}
and
\begin{eqnarray} \label{Regularity1Step3u}
\|\vecu\|_{\vecL^2(D)} &\lesssim& h^2 \|\vecv\|_{\vecL^2(D)}+h^2 \|\vecf\|_{\vecL^2(D)}+h^4 \|\vecg\|_{\vecL^2(D)}+ h^5 \|\mdiv \vecg\|_{L^2(D)}  \nonumber \\
&+&h^2 \|\mdiv \vecf\|_{L^2(D)}  +h^{\frac{7}{2}} |\vecJ(\vecv_T)|_{\vecH^{-\frac{3}{2}}_{sc}(\Gamma)} + h^{\frac{7}{2}} |\gamma \vecv|_{\vecH^{-\frac{1}{2}}_{sc}(\Gamma)} .
\end{eqnarray}
\end{lemma}
\begin{proof}
The assumptions in Lemma \ref{1stRegularityuLemma} and Lemma \ref{1stRegularityvNLemma} are satisfied. Therefore we substitue estimates (\ref{Regularityu}) and (\ref{RegularityDivu}) into estimate (\ref{RegularityvNormal}) to get
\begin{eqnarray*}
|\vecv_N|_{\vecH^{-\frac{1}{2}}_{sc}(\Gamma)} &\lesssim & h^{\frac{3}{2}} \|\vecg\|_{\vecL^2(D)} +  h^{\frac{5}{2}}  \|\mdiv \vecg\|_{L^2(D)}  + h^{-\frac{1}{2}} \|\vecf\|_{\vecL^2(D)} + h^{\frac{1}{2}}\|\vecv\|_{\vecL^2(D)} \\
&+& h^{-\frac{1}{2}}\|\mdiv \vecf\|_{L^2(D)} + h|\vecv_N|_{\vecH^{-\frac{1}{2}}_{sc}(\Gamma)} + h |\vecJ(\vecv_T)|_{\vecH^{-\frac{3}{2}}_{sc}(\Gamma)} + h |\gamma \vecv|_{\vecH^{-\frac{1}{2}}_{sc}(\Gamma)} .
\end{eqnarray*}
Since $\vecv_N \in \vecH^{-\frac{1}{2}}_{sc}(\Gamma)$, for $h$ small enough we get estimate (\ref{Regularity1Step3vN}). Inequality (\ref{Regularityu}) then yields estimate (\ref{Regularity1Step3u}). This proves the lemma.
\end{proof} \proofend

\begin{lemma}  \label{1stRegularityuNLemma}
Assume that assumption \ref{Assumptionmconstant} holds. Assume in addition that $|\xi|^2-\mu \not = 0$ , $|\xi|^2- (1+m)\mu \not = 0$ for any $\xi$ and $x \in \overline{D}$, and $R_0(x,\xi')-\frac{1+m}{2+m}\mu \not =0 $ for any $\xi'$ and $x \in \Gamma$.  Then for sufficiently small $h$
\begin{eqnarray} \label{Section2Step4Eqnu}
|\vecu_N|_{\vecH^{\frac{3}{2}}_{sc}(\Gamma)}  &\lesssim & h^{\frac{7}{2}} \|\vecg\|_{\vecL^2(D)} +h^{\frac{9}{2}}  \|\mdiv \vecg\|_{\scH{1}(D)} + h^{\frac{3}{2}} \|\vecf\|_{\vecL^2(D)} + h^{\frac{5}{2}}\|\vecv\|_{\vecL^2(D)} \nonumber \\
&+& h^{\frac{3}{2}}\|\mdiv \left( (1+m)\vecf \right) \|_{\scH{1}(D)} + h^3 |\vecJ(\vecv_T)|_{\vecH^{-\frac{3}{2}}_{sc}(\Gamma)} + h^3 |\gamma \vecv|_{\vecH^{-\frac{1}{2}}_{sc}(\Gamma)} .
\end{eqnarray}
\end{lemma}
\begin{proof}
From equation (\ref{Section2Substep2Eqnu}) we have
\begin{eqnarray*}
\mop(\frac{\lambda_2}{\lambda_2-\lambda_1})\vecu_N&=& h^2 \mop(\frac{m(\rho_2-\rho_1+\lambda_2-\lambda_1)}{(\lambda_1-\lambda_2)(\lambda_1-\rho_2)(\rho_1-\lambda_2)(\rho_1-\rho_2)}) (-\nu \mdiv\!^h_\Gamma \vecv_T)   \\
 &+&  h^2 \mop(\frac{m(\rho_2\lambda_2-\rho_1\lambda_1)}{(\lambda_1-\lambda_2)(\lambda_1-\rho_2)(\rho_1-\lambda_2)(\rho_1-\rho_2)})\vecv_N  + \vecg_2 .
\end{eqnarray*}
Applying $\lambda_2-\lambda_1$ to both sides and combining this with equation (\ref{Section2Substep1Eqnv}) yields
\begin{eqnarray*}
\mop(\lambda_2)\vecu_N&=& h^2 \mop(r_{-2}) (-\mop(\rho_2)\vecv_N+\vecg_1)   \\
 &+&  h^2 \mop(r_{-1})\vecv_N  + \mop(r_{1})\vecg_2 + h\mop(r_{0}) \vecu_N .
\end{eqnarray*}
Since $\lambda_2 \not= 0$, for small enough $h$ we have that
\begin{eqnarray*}
|\vecu_N|_{\vecH^{\frac{3}{2}}_{sc}(\Gamma)} &\lesssim& h^2 |\vecv_N|_{\vecH^{-\frac{1}{2}}_{sc}(\Gamma)} +  h^2|\vecg_1|_{\vecH^{-\frac{1}{2}}_{sc}(\Gamma)} + |\vecg_2|_{\vecH^{\frac{3}{2}}_{sc}(\Gamma)} .
\end{eqnarray*}
Then a direct calculation (see the Calculation subsection \ref{RegularityCalculation}) yields the lemma.
\end{proof} \proofend

Now we are ready to prove the main theorem.
\begin{theorem} \label{Regularity1} 
Assume that assumption \ref{Assumptionmconstant} holds. Assume in addition that $|\xi|^2-\mu \not = 0$ , $|\xi|^2- (1+m)\mu \not = 0$ for any $\xi$ and $x \in \overline{D}$ and $R_0(x,\xi')-\frac{1+m}{2+m}\mu \not =0 $ for any $\xi'$ and $x \in \Gamma$.  Then for sufficiently small $h$
\begin{eqnarray*}
&&\|\vecv\|_{\vecL^2(D)} \lesssim h^2 \|\vecg\|_{\vecL^2(D)}+h^3 \|\mdiv \vecg\|_{\scH{1}(D)} +  \|\vecf\|_{\vecL^2(D)} +  \|\mdiv \left( (1+m)\vecf \right) \|_{\scH{1}(D)} , \\
&&\|\vecu\|_{\scH{2}(D)} \lesssim h^2 \|\vecf\|_{\vecL^2(D)} +  h^4\|\vecg\|_{\vecL^2(D)} + h^5 \|\mdiv \vecg\|_{\scH{1}(D)}  + h^2 \|\mdiv \left( (1+m)\vecf \right) \|_{\scH{1}(D)}.
\end{eqnarray*}
\end{theorem}
\begin{proof}
From (\ref{Section2Substep1Eqnv1}) we have that
\begin{eqnarray*}
\vecv &=& h \gamma K_{-M} \uvecv + h^2 \gamma \tilde{Q} \underline{\vecg} - h^3\tilde{Q} (\frac{1}{i \mu} \nabla_h \underline{\mdiv \vecg})   \nonumber \\
&+& \mop(\frac{1}{\rho_1-\rho_2}) \vecJ(\vecv_T)+ \mop(\frac{\rho_1}{\rho_1-\rho_2}) \vecv + h\mop(r_{-2}) \vecJ(\vecv_T) + h\mop(r_{-1}) \vecv .
\end{eqnarray*}
Then 
\begin{eqnarray}
\vecJ(\vecv_T) + \mop(\rho_2)\vecv &=& \mop(r_{1})( h \gamma_N K_{-M} \uvecv + h^2 \gamma_N \tilde{Q} \underline{\vecg} -h^3\gamma_N \tilde{Q} (\frac{1}{i \mu} \nabla_h \underline{\mdiv \vecg})   ) \nonumber \\
&+& h\mop(r_{-1}) \vecJ(\vecv_T) + h\mop(r_{0})\vecv:= \vecg_4 .  \label{Section2Step4Eqnv}
\end{eqnarray}
From (\ref{Section2Substep2Eqnu1}) we have that
\begin{eqnarray*}
\vecu_N&=& h \gamma K_{-M} \uvecu + h^3 \gamma Q(m K_{-M}\uvecv) + h^4 \gamma Qm\tilde{Q} \underline{\vecg} - h^5 \gamma Qm\tilde{Q} (\frac{1}{i \mu} \nabla_h \underline{\mdiv \vecg})  \\
&+& h^2 \gamma Q((1+m)\underline{\vecf}) + \frac{h}{i}\gamma Q( \nabla_h \underline{\mdiv \vecu}) \nonumber \\
&+& h^2 \mop(\frac{m(\rho_2-\rho_1+\lambda_2-\lambda_1)}{(\lambda_1-\lambda_2)(\lambda_1-\rho_2)(\rho_1-\lambda_2)(\rho_1-\rho_2)}) \vecJ(\vecv_T) \nonumber \\
 &+&  h^2 \mop(\frac{m(\rho_2\lambda_2-\rho_1\lambda_1)}{(\lambda_1-\lambda_2)(\lambda_1-\rho_2)(\rho_1-\lambda_2)(\rho_1-\rho_2)}) \vecv \nonumber \\
&+&  \mop(\frac{1}{\lambda_1-\lambda_2}) \nabla_\Gamma^h(\vecu_N \cdot \nu) +  \mop(\frac{\lambda_1}{\lambda_1-\lambda_2})\vecu_N \\
&+& h^3\mop(r_{-4})\vecJ(\vecv_T) + h^3\mop(r_{-3}) \vecv+h \mop(r_{-2})\nabla_\Gamma^h(\vecu_N \cdot \nu)  +  h\mop(r_{-1})\vecu_N .
\end{eqnarray*}
Combining the above with equation (\ref{Section2Step4Eqnv}) yields
\begin{eqnarray*}
&&\mop(-\frac{m\rho_2(\rho_2-\rho_1+\lambda_2-\lambda_1)}{(\lambda_1-\lambda_2)(\lambda_1-\rho_2)(\rho_1-\lambda_2)(\rho_1-\rho_2)}+ \frac{m(\rho_2\lambda_2-\rho_1\lambda_1)}{(\lambda_1-\lambda_2)(\lambda_1-\rho_2)(\rho_1-\lambda_2)(\rho_1-\rho_2)}) \vecv \\
=&& -h^{-2} \left( h \gamma K_{-M} \uvecu + h^3 \gamma Q(m K_{-M}\uvecv) + h^4 \gamma Qm\tilde{Q} \underline{\vecg} - h^5 \gamma Qm\tilde{Q} (\frac{1}{i \mu} \nabla_h \underline{\mdiv \vecg}) \right)\\
&-& h^{-2}\left( h^2 \gamma Q((1+m)\underline{\vecf}) + \frac{h}{i}\gamma Q( \nabla_h \underline{\mdiv \vecu}) \right) \\
+&& h^{-2}\mop(r_{-1}) \nabla_\Gamma^h(\vecu_N \cdot \nu)  +  h^{-2}\mop(r_{0})\vecu_N + \mop(r_{-3})\vecg_4 + h\mop(r_{-4}) \vecJ(\vecv_T) + h\mop(r_{-3}) \vecv \\
:=&& \vecg_5.
\end{eqnarray*}
As in \cite{R1}, the symbol 
$$
-\frac{m\rho_2(\rho_2-\rho_1+\lambda_2-\lambda_1)}{(\lambda_1-\lambda_2)(\lambda_1-\rho_2)(\rho_1-\lambda_2)(\rho_1-\rho_2)}+ \frac{m(\rho_2\lambda_2-\rho_1\lambda_1)}{(\lambda_1-\lambda_2)(\lambda_1-\rho_2)(\rho_1-\lambda_2)(\rho_1-\rho_2)}
$$
is not zero and we can apply its parametrix to the above equation. Then
\begin{eqnarray} \label{sec2Regularity1Estimatevbyg5}
|\vecv|_{\vecH^{-\frac{1}{2}}_{sc}(\Gamma)} \lesssim  |\vecg_5|_{\vecH^{\frac{3}{2}}_{sc}(\Gamma)} .
\end{eqnarray}
Estimates (\ref{Section2Step4Eqnv}) and (\ref{Regularity1Step5Estimateg4}) yields
\begin{eqnarray} \label{Regularity1Step5EstimateVecJ}
|\vecJ(\vecv_T)|_{\vecH^{-\frac{3}{2}}_{sc}(\Gamma)} &\lesssim& |\vecv|_{\vecH^{-\frac{1}{2}}_{sc}(\Gamma)}  +|\vecg_4|_{\vecH^{-\frac{3}{2}}_{sc}(\Gamma)} \nonumber \\
&\lesssim& h^{\frac{1}{2}}\|\vecv\|_{\vecL^2(D)} + h^{\frac{3}{2}}\|\vecg\|_{\vecL^2(D)} +h^{\frac{5}{2}}\|\mdiv \vecg\|_{\scH{1}(D)} \nonumber \\
&+& h | \vecJ(\vecv_T) |_{\vecH^{-\frac{3}{2}}_{sc}(\Gamma)} + |\vecv |_{\vecH^{-\frac{1}{2}}_{sc}(\Gamma)}. 
\end{eqnarray}
A direct calculation (see the Calculation subsection \ref{RegularityCalculation}) yields for small enough $h$
\begin{eqnarray}\label{SectionRegularity1Step4VJ}
&& |\vecv|_{\vecH^{-\frac{1}{2}}_{sc}(\Gamma)} + |\vecJ(\vecv_T)|_{\vecH^{-\frac{3}{2}}_{sc}(\Gamma)} \nonumber \\
&\lesssim& h^{\frac{1}{2}}\|\vecv\|_{\vecL^2(D)} + h^{\frac{3}{2}}\|\vecg\|_{\vecL^2(D)} +h^{\frac{5}{2}}\|\mdiv \vecg\|_{\scH{1}(D)} \nonumber  \\
&+& h^{-\frac{1}{2}}\|\vecf\|_{\vecL^2(D)} + h^{-\frac{1}{2}}\|\mdiv \left( (1+m)\vecf \right)\|_{\scH{1}(D)}. 
\end{eqnarray}
Notice that $\vecv$ satisfies equation (\ref{Section2SubStep1Eqnv}). Then
applying estimates (\ref{PreliminaryResultsBounaryToDomain}) and
(\ref{SectionRegularity1Step4VJ}) gives
\begin{eqnarray} \label{Section2Step4Estimatev}
\|\vecv\|_{\vecL^2(D)} \lesssim h^2 \|\vecg\|_{\vecL^2(D)} + h^3 \|\mdiv \vecg\|_{\scH{1}(D)} +  \|\vecf\|_{\vecL^2(D)} +  \|\mdiv \left( (1+m)\vecf \right) \|_{H^1(D)}.
\end{eqnarray}
From equation (\ref{Section2Step1Eqnu}) we have that
\begin{eqnarray*}
\|\vecu\|_{\scH{2}(D)} \lesssim h^2 \|\vecv\|_{\vecL^2(D)} +  h^2\|\vecf\|_{\vecL^2(D)} + h\|\mdiv \vecu\|_{\scH{1}(D)}+ h^{\frac{1}{2}} |\vecu_N|_{\vecH^{\frac{3}{2}}_{sc}(\Gamma)} .
\end{eqnarray*}
From estimates (\ref{RegularityDivuH1}) (\ref{Section2Step4Eqnu}) (\ref{SectionRegularity1Step4VJ}) and (\ref{Section2Step4Estimatev}) we have 
\begin{eqnarray*}
\|\vecu\|_{\scH{2}(D)} \lesssim h^2 \|\vecf\|_{\vecL^2(D)} +  h^4\|\vecg\|_{\vecL^2(D)} + h^5 \|\mdiv \vecg\|_{\scH{1}(D)} + h^2\|\mdiv \left( (1+m)\vecf \right) \|_{\scH{1}(D)}.
\end{eqnarray*}
This completes the proof.
\end{proof} \proofend
\subsection{Calculation} \label{RegularityCalculation}
In this subsection, we will show the necessary calculations for subsection \ref{SectionRegularity1}.

1. {Calculation for Lemma \ref{1stRegularityuLemma}} 

\medskip

Taking the divergence of equation (\ref{FormulationEu}) and noticing that $\lambda=\mu h^{-2}$ yields
\begin{eqnarray} \label{Section2Step2Divu}
-\mu((1+m)\mdiv \vecu+\nabla m \cdot \vecu)-h^2 (\nabla m \cdot \vecv+m\mdiv \vecv)= h^2\mdiv\left( \left(1+m\right) \vecf\right).
\end{eqnarray}
Since $\nabla m$ has compact support in $D$ and $|\xi|^2-\mu \not = 0$, estimate (\ref{PreliminaryResultsvBoundary}) yields
\begin{eqnarray*}
\|\nabla m \cdot \vecv\|_{L^2(D)} \lesssim h \|\vecv\|_{\vecL^2(D)}+ h^2 \|\vecg\|_{\vecL^2(D)} + h^3 \|\mdiv \vecg\|_{L^2(D)}.
\end{eqnarray*}
Therefore
\begin{eqnarray} \label{RegularityDivuPriori}
\|\mdiv \vecu\|_{L^2(D)} &\lesssim& \|\vecu\|_{\vecL^2(D)}+ h^3 \|\vecv\|_{\vecL^2(D)} +h^2 \|\mdiv \vecv\|_{\vecL^2(D)}+ h^4 \|\vecg\|_{\vecL^2(D)}+h^5 \|\mdiv \vecg\|_{L^2(D)}  \nonumber\\
&+& h^2 \|\mdiv \left((1+m)\vecf \right)\|_{L^2(D)} .
\end{eqnarray}
Since $-\lambda \mdiv \vecv=\mdiv \vecg$, we have that
\begin{eqnarray*} 
\mu \mdiv \vecv = -h^2 \mdiv \vecg
\end{eqnarray*}
and therefore 
\begin{eqnarray}  \label{RegularityDivvDivg}
\|\mdiv \vecv\|_{\scH{s}(D)} \lesssim h^2 \| \mdiv \vecg\|_{\scH{s}(D)} .
\end{eqnarray} 
Substituting (\ref{RegularityDivvDivg}) (with s=0) into (\ref{RegularityDivuPriori}) yields
\begin{eqnarray} \label{RegularityDivu}
\|\mdiv \vecu\|_{L^2(D)} &\lesssim& \|\vecu\|_{\vecL^2(D)}+ h^3 \|\vecv\|_{\vecL^2(D)} + h^4 \|\vecg\|_{\vecL^2(D)}+h^4 \|\mdiv \vecg\|_{L^2(D)}  \nonumber\\
&+& h^2 \|\mdiv \left((1+m)\vecf \right)\|_{L^2(D)} .
\end{eqnarray}
Notice that since $|\vecf_N|_{\vecH^{-\frac{1}{2}}(\Gamma)} \lesssim \|\vecf\|_{\vecL^2(D)} +\|\mdiv \vecf\|_{L^2(D)}$, then
\begin{eqnarray} \label{Regularity1EstimatefN}
|\vecf_N|_{\vecH^{-\frac{1}{2}}_{sc}(\Gamma)} \lesssim h^{-\frac{1}{2}} \left( \|\vecf\|_{\vecL^2(D)} + \|\mdiv \vecf\|_{L^2(D)} \right).
\end{eqnarray}
From equation (\ref{VectorialPotentialsutov}) and estimate (\ref{Regularity1EstimatefN}) we have that
\begin{eqnarray} \label{Section2Step2EstimateuNbyVNfN}
| \vecu_N|_{\vecH^{-\frac{1}{2}}_{sc}(\Gamma)} 
&\lesssim& h^2| \vecv_N|_{\vecH^{-\frac{1}{2}}_{sc}(\Gamma)} + h^2 |\vecf_N|_{\vecH^{-\frac{1}{2}}_{sc}(\Gamma)} \nonumber \\
&\lesssim& h^2| \vecv_N|_{\vecH^{-\frac{1}{2}}_{sc}(\Gamma)} + h^{\frac{3}{2}} \left( \|\vecf\|_{\vecL^2(D)} + \|\mdiv \vecf\|_{L^2(D)} \right) .
\end{eqnarray}
Plugging estimates (\ref{RegularityDivu}) and (\ref{Section2Step2EstimateuNbyVNfN})  into (\ref{Section2Step2Estimateu}) yields for $h$ small enough
\begin{eqnarray*} 
\|\vecu\|_{\vecL^2(D)} &\lesssim& h^2 \|\vecv\|_{\vecL^2(D)}+h^2 \|\vecf\|_{\vecL^2(D)}+h^5 \|\vecg\|_{\vecL^2(D)}+ h^5 \|\mdiv \vecg\|_{L^2(D)}  \nonumber \\
&+&h^2 \|\mdiv \vecf\|_{L^2(D)} +  h^{\frac{5}{2}}|\vecv_N|_{\vecH^{-\frac{1}{2}}_{sc}(\Gamma)} .
\end{eqnarray*}

2. Calculation for Lemma \ref{1stRegularityvNLemma} 

\medskip

\textit{Calculation for Step 1}

\medskip

Taking the traces on the boundary $\Gamma$ and using equations (\ref{PreliminaryResultsTildeQBoundary1})-(\ref{PreliminaryResultsTildeQBoundary2}) we have
\begin{eqnarray}
\gamma \vecv &=& h \gamma K_{-M} \uvecv + h^2 \gamma \tilde{Q} \underline{\vecg} - h^3 \gamma \tilde{Q} (\frac{1}{i \mu} \nabla_h \underline{\mdiv \vecg})  \nonumber \\
&+& \mbox{op}(\frac{1}{\rho_1-\rho_2}) \vecJ(\vecv_T) + \mbox{op}(\frac{\rho_1}{\rho_1-\rho_2}) \vecv  +h \mbox{op}(r_{-2}) \vecJ(\vecv_T) + h\mop(r_{-1}) \vecv   \label{Section2Substep1Eqnv1}
\end{eqnarray}
where $\gamma$ is the trace operator on $\Gamma$. Furthermore taking the normal component yields
\begin{eqnarray*}
\vecv_N &=& h \gamma_N K_{-M} \uvecv + h^2 \gamma_N \tilde{Q} \underline{\vecg} - h^3\gamma_N \tilde{Q} (\frac{1}{i \mu} \nabla_h \underline{\mdiv \vecg})  \\
&+& \mop(\frac{1}{\rho_1-\rho_2})(- \nu \mdiv\!^h_\Gamma \vecv_T)+ \mop(\frac{\rho_1}{\rho_1-\rho_2}) \vecv_N + h\mop(r_{-2}) \vecJ(\vecv_T) + h\mop(r_{-1}) \vecv .
\end{eqnarray*}
Applying $\mop(\rho_2-\rho_1)$ to both sides yields equation (\ref{Section2Substep1Eqnv}).

\medskip

\textit{Calculation for Step 2}

\medskip

Substituting equation (\ref{Section2SubStep1Eqnv}) into equation (\ref{Section2Step1Eqnu}) yields
\begin{eqnarray*}
\vecu &=&  h K_{-M} \uvecu + h^3 Q(m K_{-M}\uvecv) + h^4Qm\tilde{Q} \underline{\vecg} - h^5Qm\tilde{Q} (\frac{1}{i \mu} \nabla_h \underline{\mdiv \vecg})   \\
&+& h^2Q((1+m)\underline{\vecf}) + \frac{h}{i}Q( \nabla_h \underline{\mdiv \vecu}) \nonumber \\
&+& h^2 Qm\tilde{Q}[\frac{h}{i} \vecJ(\vecv_T) \otimes \delta_{s=0} + \frac{h}{i}\vecv \otimes D^h_s \delta_{s=0}] \nonumber \\
&+&  Q(\frac{h}{i}\nabla_\Gamma^h(\vecu_N \cdot \nu) \otimes \delta_{s=0}) +  Q(\frac{h}{i}\vecu_N \otimes D^h_s \delta_{s=0}) . \nonumber
\end{eqnarray*}
Taking the traces on $\Gamma$ and using equations (\ref{PreliminaryResultsQmTildeQBoundary1}) (\ref{PreliminaryResultsQmTildeQBoundary2}) yields
\begin{eqnarray}
\vecu|_\Gamma&=& h \gamma K_{-M} \uvecu + h^3 \gamma Q(m K_{-M}\uvecv) + h^4 \gamma Qm\tilde{Q} \underline{\vecg} -  h^5\gamma Qm\tilde{Q} (\frac{1}{i \mu} \nabla_h \underline{\mdiv \vecg})   \nonumber \\
&+& h^2 \gamma Q((1+m)\underline{\vecf}) + \frac{h}{i}\gamma Q( \nabla_h \underline{\mdiv \vecu}) \nonumber \\
&+& h^2 \mop \left( \frac{m(\rho_2-\rho_1+\lambda_2-\lambda_1)}{(\lambda_1-\lambda_2)(\lambda_1-\rho_2)(\rho_1-\lambda_2)(\rho_1-\rho_2)} \right) \vecJ(\vecv_T) \nonumber \\
 &+&  h^2 \mop \left( \frac{m(\rho_2\lambda_2-\rho_1\lambda_1)}{(\lambda_1-\lambda_2)(\lambda_1-\rho_2)(\rho_1-\lambda_2)(\rho_1-\rho_2)} \right) \vecv \nonumber \\
&+&  \mop(\frac{1}{\lambda_1-\lambda_2})\nabla_\Gamma^h(\vecu_N \cdot \nu) +  \mop(\frac{\lambda_1}{\lambda_1-\lambda_2})\vecu_N \nonumber \\
&+& h^3\mop(r_{-4}) \vecJ(\vecv_T) + h^3\mop(r_{-3})\vecv+ h\mop(r_{-2})\nabla_\Gamma^h(\vecu_N \cdot \nu) +  h\mop(r_{-1})\vecu_N .
\label{Section2Substep2Eqnu1}
\end{eqnarray}
Taking the normal component and noticing that $\nu \cdot \nabla_\Gamma^h(\vecu_N \cdot \nu)=0$ yields equation (\ref{Section2Substep2Eqnu}).

\medskip

\textit{Calculation for Step 4}

\medskip

Applying estimates (\ref{PreliminaryResultsEstimateQ}) and (\ref{PrelinimaryResultsTrace}) gives
\begin{eqnarray*}
|\vecg_2|_{\vecH^{-\frac{1}{2}}_{sc}(\Gamma)} 
&\lesssim& h^{\frac{1}{2}}\|\vecu\|_{\vecL^2(D)}+h^{\frac{5}{2}}\|\vecv\|_{\vecL^2(D)} + h^{\frac{7}{2}}\|\vecg\|_{\vecL^2(D)} +h^{\frac{9}{2}}\|\mdiv \vecg\|_{\vecL^2(D)} + h^{\frac{3}{2}}\|\vecf\|_{\vecL^2(D)} \\
&+& h^{\frac{1}{2}}\|\mdiv \vecu\|_{\vecL^2(D)} +  h^3 | \vecJ(\vecv_T) |_{\vecH^{-\frac{3}{2}}_{sc}(\Gamma)} + h^3 | \vecv |_{\vecH^{-\frac{1}{2}}_{sc}(\Gamma)} + h | \vecu_N |_{\vecH^{-\frac{1}{2}}_{sc}(\Gamma)}.
\end{eqnarray*}
From equation (\ref{VectorialPotentialsutov}) $\vecu_N=-h^2
\frac{m}{\mu(1+m)}\vecv_N - h^2 \frac{1}{\mu} \vecf_N$, and therefore
\begin{eqnarray}
|\vecg_2|_{\vecH^{-\frac{1}{2}}_{sc}(\Gamma)} 
&\lesssim& h^{\frac{1}{2}}\|\vecu\|_{\vecL^2(D)}+h^{\frac{5}{2}}\|\vecv\|_{\vecL^2(D)} + h^{\frac{7}{2}}\|\vecg\|_{\vecL^2(D)} +h^{\frac{9}{2}}\|\mdiv \vecg\|_{\vecL^2(D)} + h^{\frac{3}{2}}\|\vecf\|_{\vecL^2(D)} \nonumber \\
&+& h^{\frac{1}{2}}\|\mdiv \vecu\|_{\vecL^2(D)} +  h^3 | \vecJ(\vecv_T) |_{\vecH^{-\frac{3}{2}}_{sc}(\Gamma)} + h^3 | \vecv |_{\vecH^{-\frac{1}{2}}_{sc}(\Gamma)} \nonumber \\
&+& h^3 | \vecv_N|_{\vecH^{-\frac{1}{2}}_{sc}(\Gamma)} + h^3 |\vecf_N|_{\vecH^{-\frac{1}{2}}_{sc}(\Gamma)}. \label{Regularity1Estimateg2}
\end{eqnarray}
Applying estimates (\ref{PreliminaryResultsEstimateQ}) and
(\ref{PrelinimaryResultsTrace}) yield
\begin{eqnarray*}
|\vecg_1|_{\vecH^{-\frac{1}{2}}_{sc}(\Gamma)}
&\lesssim& h^{\frac{1}{2}}\|\vecv\|_{\vecL^2(D)}+ h^{\frac{3}{2}}\|\vecg\|_{\vecL^2(D)} +h^{\frac{5}{2}}\|\mdiv \vecg\|_{\vecL^2(D)} \\
&+&  h | \vecJ(\vecv_T) |_{\vecH^{-\frac{3}{2}}_{sc}(\Gamma)} + h |\vecv|_{\vecH^{-\frac{1}{2}}_{sc}(\Gamma)} + h | \vecv_N|_{\vecH^{-\frac{1}{2}}_{sc}(\Gamma)} + h |\nu \mdiv\!^h_\Gamma \vecv_T|_{\vecH^{-\frac{3}{2}}_{sc}(\Gamma)} .
\end{eqnarray*}
Since $\vecv_N$ and $-\nu \mdiv^h_\Gamma \vecv_T$ are the normal components of $\vecv$ and $\vecJ(\vecv_T)$ respectively, then
\begin{eqnarray}
|\vecg_1|_{\vecH^{-\frac{1}{2}}_{sc}(\Gamma)} \nonumber 
&\lesssim& h^{\frac{1}{2}}\|\vecv\|_{\vecL^2(D)}+ h^{\frac{3}{2}}\|\vecg\|_{\vecL^2(D)} +h^{\frac{5}{2}}\|\mdiv \vecg\|_{\vecL^2(D)} \nonumber \\
&+&  h | \vecJ(\vecv_T) |_{\vecH^{-\frac{3}{2}}_{sc}(\Gamma)} + h |\vecv|_{\vecH^{-\frac{1}{2}}_{sc}(\Gamma)} . \label{Regularity1Estimateg1}
\end{eqnarray}
Then estimates (\ref{Regularity1EstimatefN}) (\ref{Regularity1Estimateg2}) and (\ref{Regularity1Estimateg1}) yield for small enough $h$ that
\begin{eqnarray*} 
|\vecv_N|_{\vecH^{-\frac{1}{2}}_{sc}(\Gamma)} &\lesssim& h^{\frac{3}{2}} \|\vecg\|_{\vecL^2(D)} + h^{\frac{5}{2}} \|\mdiv \vecg\|_{\vecL^2(D)} + h^{-\frac{1}{2}} \|\vecf\|_{\vecL^2(D)} + h^{-\frac{1}{2}} \|\mdiv \vecf\|_{L^2(D)} \nonumber \\
&+& h^{\frac{1}{2}}\|\vecv\|_{\vecL^2(D)} + h^{-\frac{3}{2}}\|\vecu\|_{\vecL^2(D)} + h^{-\frac{3}{2}}\|\mdiv \vecu\|_{L^2(D)} \nonumber \\
&+& h |\vecJ(\vecv_T)|_{\vecH^{-\frac{3}{2}}_{sc}(\Gamma)} + h |\gamma \vecv|_{\vecH^{-\frac{1}{2}}_{sc}(\Gamma)} .
\end{eqnarray*}

\medskip

3. Calculation for Lemma \ref{1stRegularityuNLemma} 

\medskip

From inequalities (\ref{PreliminaryResultsEstimateQ}) and
(\ref{PrelinimaryResultsTrace}) one get
\begin{eqnarray} \label{Regularity1Step3estimateg2}
|\vecg_2|_{\vecH^{\frac{3}{2}}_{sc}(\Gamma)} 
&\lesssim& h^{\frac{1}{2}}\|\vecu\|_{\vecL^2(D)}+h^{\frac{5}{2}}\|\vecv\|_{\vecL^2(D)} + h^{\frac{7}{2}}\|\vecg\|_{\vecL^2(D)} +h^{\frac{9}{2}}\|\mdiv \vecg\|_{\vecL^2(D)} + h^{\frac{3}{2}}\|\vecf\|_{\vecL^2(D)} \nonumber \\
&+& h|\gamma_N Q \nabla_h \underline{\mdiv \vecu}|_{\vecH^{\frac{3}{2}}_{sc}(\Gamma)}  +  h^3 | \vecJ(\vecv_T) |_{\vecH^{-\frac{3}{2}}_{sc}(\Gamma)} + h^3 | \vecv |_{\vecH^{-\frac{1}{2}}_{sc}(\Gamma)} + h | \vecu_N |_{\vecH^{\frac{1}{2}}_{sc}(\Gamma)}.
\end{eqnarray}
This motivates us to derive an estimate for $|\gamma_N Q \nabla_h \underline{\mdiv \vecu}|_{\vecH^{\frac{3}{2}}_{sc}(\Gamma)}$. Since $\nabla m$ has compact support in $D$, then estimate (\ref{PreliminaryResultsvBoundary}) yields
\begin{eqnarray*}
\|\nabla m \cdot \vecv\|_{\scH{1}(D)} \lesssim h \|\vecv\|_{\vecL^2(D)}+ h^2 \|\vecg\|_{\vecL^2(D)} +h^3 \|\mdiv \vecg\|_{L^2(D)}
\end{eqnarray*}
and 
\begin{eqnarray*}
\|\nabla m \cdot \vecu\|_{\scH{1}(D)} \lesssim  h \|\vecu\|_{\vecL^2(D)}+h^2 \|\vecv\|_{\vecL^2(D)}+ h^2 \|\vecf\|_{\vecL^2(D)} +h\|\mdiv \vecu\|_{L^2(D)} .
\end{eqnarray*}
From equation (\ref{Section2Step2Divu}) and estimate (\ref{RegularityDivvDivg}) (with s=1) we have that for small $h$
\begin{eqnarray} \label{RegularityDivuH1}
\|\mdiv \vecu\|_{\scH{1}(D)} &\lesssim& \|\nabla m \cdot \vecu\|_{\scH{1}(D)} + h^2\|\nabla m \cdot \vecv\|_{\scH{1}(D)}  + h^2\|\mdiv\vecv\|_{\scH{1}(D)} \nonumber \\
&+& h^2 \|\mdiv \left( (1+m)\vecf \right)\|_{\scH{1}(D)} \nonumber \\
&\lesssim& h^2 \|\vecv\|_{\vecL^2(D)}+ h^4 \|\vecg\|_{\vecL^2(D)} + h^4 \|\mdiv \vecg\|_{\scH{1}(D)} +h \|\vecu\|_{\vecL^2(D)} \nonumber \\
&+&  h^2 \|\mdiv \left( (1+m)\vecf \right)\|_{\scH{1}(D)}+h^2 \|\vecf\|_{\vecL^2(D)} .
\end{eqnarray}
From Lemma \ref{PreliminaryResultsQNablau} and estimate (\ref{PrelinimaryResultsTrace}) we have
\begin{eqnarray*} 
|\gamma_N Q \nabla_h \underline{\mdiv \vecu}|_{\vecH^{\frac{3}{2}}_{sc}(\Gamma)} &\lesssim& h^{-\frac{1}{2}}\|Q \nabla_h \underline{\mdiv \vecu}\|_{\scH{2}(D)} \nonumber \\
&\lesssim& h^{-\frac{1}{2}}\|\mdiv \vecu\|_{\scH{1}(D)} .
\end{eqnarray*}
Combined with  (\ref{RegularityDivuH1}), this inequality gives
\begin{eqnarray} 
|\gamma_N Q \nabla_h \underline{\mdiv \vecu}|_{\vecH^{\frac{3}{2}}_{sc}(\Gamma)} 
&\lesssim& h^{\frac{3}{2}} \|\vecv\|_{\vecL^2(D)}+ h^{\frac{7}{2}} \|\vecg\|_{\vecL^2(D)} + h^{\frac{7}{2}} \|\mdiv \vecg\|_{\scH{1}(D)} + h^{\frac{1}{2}} \|\vecu\|_{\vecL^2(D)}\nonumber \\
&+& h^{\frac{3}{2}} \|\mdiv \left( (1+m)\vecf \right)\|_{\scH{1}(D)}+h^{\frac{3}{2}} \|\vecf\|_{\vecL^2(D)} . \label{RegularityStep3EstimatetraceQnabladivu}
\end{eqnarray}
Substituting estimates (\ref{Regularity1Step3u}) and (\ref{RegularityStep3EstimatetraceQnabladivu}) into (\ref{Regularity1Step3estimateg2})  yields
\begin{eqnarray} \label{Regularity1Step3estimateg2more1}
|\vecg_2|_{\vecH^{\frac{3}{2}}_{sc}(\Gamma)} 
&\lesssim& h^{\frac{1}{2}}\|\vecu\|_{\vecL^2(D)}+h^{\frac{5}{2}}\|\vecv\|_{\vecL^2(D)} + h^{\frac{7}{2}}\|\vecg\|_{\vecL^2(D)} +h^{\frac{9}{2}}\|\mdiv \vecg\|_{\scH{1}(D)} + h^{\frac{3}{2}}\|\vecf\|_{\vecL^2(D)} \nonumber \\
&+& h^{\frac{5}{2}} \|\mdiv \left( (1+m)\vecf \right)\|_{\scH{1}(D)} +  h^3 | \vecJ(\vecv_T) |_{\vecH^{-\frac{3}{2}}_{sc}(\Gamma)} + h^3 | \vecv |_{\vecH^{-\frac{1}{2}}_{sc}(\Gamma)} + h | \vecu_N |_{\vecH^{\frac{1}{2}}_{sc}(\Gamma)} \nonumber \\
&\lesssim& h^{\frac{7}{2}}\|\vecg\|_{\vecL^2(D)} +h^{\frac{9}{2}}\|\mdiv \vecg\|_{\scH{1}(D)} + h^{\frac{3}{2}}\|\vecf\|_{\vecL^2(D)} +h^{\frac{5}{2}} \|\mdiv \left( (1+m)\vecf \right)\|_{\scH{1}(D)}  \nonumber \\
&+&  h^{\frac{5}{2}}\|\vecv\|_{\vecL^2(D)} + h^3 | \vecJ(\vecv_T) |_{\vecH^{-\frac{3}{2}}_{sc}(\Gamma)} + h^3 | \vecv |_{\vecH^{-\frac{1}{2}}_{sc}(\Gamma)} + h | \vecu_N |_{\vecH^{\frac{1}{2}}_{sc}(\Gamma)} .
\end{eqnarray}
Combining estimates (\ref{Regularity1Estimateg1}) (\ref{Regularity1Step3vN}) and (\ref{Regularity1Step3estimateg2more1}) implies that
\begin{eqnarray*} 
|\vecu_N|_{\vecH^{\frac{3}{2}}_{sc}(\Gamma)} &\lesssim & h^2 |\vecv_N|_{\vecH^{-\frac{1}{2}}_{sc}(\Gamma)} +  h^2|\vecg_1|_{\vecH^{-\frac{1}{2}}_{sc}(\Gamma)} + |\vecg_2|_{\vecH^{\frac{3}{2}}_{sc}(\Gamma)} \nonumber  \\
&\lesssim & h^{\frac{7}{2}} \|\vecg\|_{\vecL^2(D)} +h^{\frac{9}{2}}  \|\mdiv \vecg\|_{\scH{1}(D)} + h^{\frac{3}{2}} \|\vecf\|_{\vecL^2(D)} + h^{\frac{5}{2}}\|\vecv\|_{\vecL^2(D)} \nonumber \\
&+& h^{\frac{3}{2}}\|\mdiv \left( (1+m)\vecf \right) \|_{\scH{1}(D)} + h^3 |\vecJ(\vecv_T)|_{\vecH^{-\frac{3}{2}}_{sc}(\Gamma)} + h^3 |\gamma \vecv|_{\vecH^{-\frac{1}{2}}_{sc}(\Gamma)} .
\end{eqnarray*}

\medskip

4. Calculation for Theorem \ref{Regularity1}

\medskip

Applying estimates (\ref{PreliminaryResultsEstimateQ}) and (\ref{PrelinimaryResultsTrace}) gives
\begin{eqnarray*}
|\vecg_5|_{\vecH^{\frac{3}{2}}_{sc}(\Gamma)} 
&\lesssim& h^{-\frac{3}{2}}\|\vecu\|_{\vecL^2(D)}+h^{\frac{1}{2}}\|\vecv\|_{\vecL^2(D)} + h^{\frac{3}{2}}\|\vecg\|_{\vecL^2(D)} +h^{\frac{5}{2}}\|\mdiv \vecg\|_{\vecL^2(D)} + h^{-\frac{1}{2}}\|\vecf\|_{\vecL^2(D)} \\
&+& h^{-1}\|\gamma Q \nabla_h \underline{\mdiv \vecu}\|_{\vecH^{\frac{3}{2}}_{sc}(\Gamma)}  + h^{-2} | \vecu_N |_{\vecH^{\frac{3}{2}}_{sc}(\Gamma)} \\
&+&  h | \vecJ(\vecv_T) |_{\vecH^{-\frac{3}{2}}_{sc}(\Gamma)} + h | \vecv |_{\vecH^{-\frac{1}{2}}_{sc}(\Gamma)} + |\vecg_4|_{\vecH^{-\frac{3}{2}}_{sc}(\Gamma)}.
\end{eqnarray*}
Applying estimates (\ref{PreliminaryResultsEstimateQ}) and (\ref{PrelinimaryResultsTrace}) gives
\begin{eqnarray} \label{Regularity1Step5Estimateg4}
|\vecg_4|_{\vecH^{-\frac{3}{2}}_{sc}(\Gamma)} 
&\lesssim& h^{\frac{1}{2}}\|\vecv\|_{\vecL^2(D)} + h^{\frac{3}{2}}\|\vecg\|_{\vecL^2(D)} +h^{\frac{5}{2}}\|\mdiv \vecg\|_{\vecL^2(D)} \nonumber \\
&+& h | \vecJ(\vecv_T) |_{\vecH^{-\frac{3}{2}}_{sc}(\Gamma)} + h | \vecv |_{\vecH^{-\frac{1}{2}}_{sc}(\Gamma)}. 
\end{eqnarray}
Combining estimates (\ref{Regularity1Step3u}) (\ref{Section2Step4Eqnu}) (\ref{sec2Regularity1Estimatevbyg5})  (\ref{RegularityStep3EstimatetraceQnabladivu}) and (\ref{Regularity1Step5Estimateg4})  yields
\begin{eqnarray}
|\vecv|_{\vecH^{-\frac{1}{2}}_{sc}(\Gamma)} 
&\lesssim& |\vecg_5|_{\vecH^{\frac{3}{2}}_{sc}(\Gamma)} \\
&\lesssim& h^{\frac{1}{2}}\|\vecv\|_{\vecL^2(D)} + h^{\frac{3}{2}}\|\vecg\|_{\vecL^2(D)} +h^{\frac{5}{2}}\|\mdiv \vecg\|_{\scH{1}(D)} + h^{-\frac{1}{2}}\|\vecf\|_{\vecL^2(D)} \nonumber  \\
&+& h^{-\frac{1}{2}}\|\mdiv \left( (1+m)\vecf \right)\|_{\scH{1}(D)}+h | \vecJ(\vecv_T) |_{\vecH^{-\frac{3}{2}}_{sc}(\Gamma)} + h | \vecv |_{\vecH^{-\frac{1}{2}}_{sc}(\Gamma)}. \label{Regularity1Step5EstimatevBoundary}
\end{eqnarray}
Combining estimates (\ref{Regularity1Step5EstimateVecJ}) and (\ref{Regularity1Step5EstimatevBoundary}) yields
\begin{eqnarray*}
&& |\vecv|_{\vecH^{-\frac{1}{2}}_{sc}(\Gamma)} + |\vecJ(\vecv_T)|_{\vecH^{-\frac{3}{2}}_{sc}(\Gamma)} \\
&\lesssim& h^{\frac{1}{2}}\|\vecv\|_{\vecL^2(D)} + h^{\frac{3}{2}}\|\vecg\|_{\vecL^2(D)} +h^{\frac{5}{2}}\|\mdiv \vecg\|_{\scH{1}(D)} + h^{-\frac{1}{2}}\|\vecf\|_{\vecL^2(D)} \nonumber  \\
&+& h^{-\frac{1}{2}}\|\mdiv \left( (1+m)\vecf \right)\|_{\scH{1}(D)}+h | \vecJ(\vecv_T) |_{\vecH^{-\frac{3}{2}}_{sc}(\Gamma)} + h | \vecv |_{\vecH^{-\frac{1}{2}}_{sc}(\Gamma)}. 
\end{eqnarray*}
Then for small enough $h$ we have estimate (\ref{SectionRegularity1Step4VJ}).
\subsection{A Second Regularity Result}
In this section we study the regularity under the restriction that
$\mdiv\left( \left( 1+m \right) \vecf \right)=0$ and $\mdiv \vecg = 0$. The
reason to consider this case is to obtain a regularizing effect of the operator $R_z$. In particular, from equation (\ref{Section2Step2Divu}), we see that
$\mdiv \vecu$ has the same regularity as $\mdiv\!\left( (1+m)\vecf \right)$
(with a similar situation for $\vecv$) and therefore the regularizing effect does not
hold in general. On the other hand, if the right hand side of equation (\ref{Section2Step2Divu}) vanishes, then the regularity of $\mdiv \vecu$ is controlled by $\vecu$ and $\nabla m \cdot \vecv$. This allows us to obtain the desired regularity of $\vecu$.
\begin{theorem} \label{Regularity2} Assume that the hypothesis of Theorem
  \ref{Regularity1} hold. If $\vecf \in \scH{2}(D)$, $\mdiv\left( \left( 1+m \right) \vecf \right)=0$ and $\mdiv \vecg=0$, then for sufficiently small $h:=\frac{1}{|\lambda|^{\frac{1}{2}}}$
\begin{eqnarray*}
&&\|\vecv\|_{\scH{2}(D)} \lesssim h^2 \|\vecg\|_{\vecL^2(D)} +  \|\vecf\|_{\scH{2}(D)}, \\
&&\|\vecu\|_{\scH{4}(D)} \lesssim h^4\|\vecg\|_{\vecL^2(D)} + h^2\|\vecf\|_{\scH{2}(D)} .
\end{eqnarray*}
Moreover if $\vecf \in \scH{4}(D)$ and $\vecg \in \scH{2}(D)$, then for sufficiently small $h:=\frac{1}{|\lambda|^{\frac{1}{2}}}$
\begin{eqnarray*}
&&\|\vecv\|_{\scH{4}(D)} \lesssim h^2 \|\vecg\|_{\scH{2}(D)} +  \|\vecf\|_{\scH{4}(D)} ,\\
&&\|\vecu\|_{\scH{6}(D)} \lesssim h^4\|\vecg\|_{\scH{2}(D)}+ h^2\|\vecf\|_{\scH{4}(D)} .
\end{eqnarray*}
\end{theorem}

\begin{proof}
We use similar arguments as in Section \ref{SectionRegularity1} and we shall
only highlight here the differences. We first prove that $\vecv \in \scH{1}(D)$ and $\vecu \in \scH{3}(D)$ if $\vecv \in \vecL^2(D)$ and $\vecu \in \scH{2}(D)$ for $\vecg \in \vecL^2(D)$ and $\vecf \in \scH{2}(D)$, then we can prove $\vecv \in \scH{2}(D)$ and $\vecu \in \scH{4}(D)$. 

\medskip

{1}. (Similarly to Lemma \ref{1stRegularityuLemma}) An a priori estimate for $\vecu$. 

\medskip

Since $\mdiv\left( \left( 1+m \right) \vecf \right)=0$ and $\mdiv \vecg=0$, Theorem \ref{Regularity1} yields
\begin{eqnarray} \label{Section6Step2Eqnu}
\|\vecu\|_{\scH{2}(D)} \lesssim h^2 \|\vecf\|_{\vecL^2(D)} +  h^4\|\vecg\|_{\vecL^2(D)}.
\end{eqnarray}

{2}. (Similarly to Lemma \ref{1stRegularityvNLemma} and Lemma \ref{1stRegularityvNuLemma}). An a priori estimate for $\vecv_N$. 

\medskip

The argument can also be divided into four steps.
Steps 1, 2 and 3 follow exactly the same way as in Section \ref{SectionRegularity1}. We
shall only indicate the changes in step 4.

\medskip

\textit{Step 4}.
From Step 4 of Lemma \ref{1stRegularityvNLemma} we have that
\begin{eqnarray*}
\mop\left( m((m+2)R-(1+m)\mu) \right) \vecv_N = h^{-2} \mop(r_2)\vecg_3 + h\mop(r_{1})\vecv_N.
\end{eqnarray*}
Then 
\begin{eqnarray*}
|\vecv_N|_{\vecH^{\frac{1}{2}}_{sc}(\Gamma)} \lesssim h^{-2}|\vecg_3|_{\vecH^{\frac{1}{2}}_{sc}(\Gamma)} + h|\vecv_N|_{\vecH^{-\frac{1}{2}}_{sc}(\Gamma)}
\end{eqnarray*}
and for $h$ small enough
\begin{eqnarray*}
|\vecv_N|_{\vecH^{\frac{1}{2}}_{sc}(\Gamma)} \lesssim h^{-2}|\vecg_3|_{\vecH^{\frac{1}{2}}_{sc}(\Gamma)} .
\end{eqnarray*}
Following the arguments in the proof of Lemma \ref{1stRegularityvNLemma}, the only
difference is to replace estimate (\ref{Regularity1EstimatefN}) by 
$$
|\vecf_N|_{\vecH^{\frac{1}{2}}_{sc}(\Gamma)} \lesssim h^{-\frac{1}{2}} \|\vecf\|_{\overline{\vecH}^1_{sc}(D)} 
$$ 
 Notice from  (\ref{SectionRegularity1Step4VJ}) and Theorem \ref{Regularity1} that
\begin{eqnarray*} 
|\vecv|_{\vecH^{-\frac{1}{2}}_{sc}(\Gamma)} \lesssim h^{\frac{3}{2}} \|\vecg\|_{\vecL^2(D)} + h^{-\frac{1}{2}} \|\vecf\|_{L^2(D)} .
\end{eqnarray*}
This gives the following estimate (corresponding to estimate (\ref{RegularityvNormal}) in Section \ref{SectionRegularity1})
\begin{eqnarray*}
|\vecv_N|_{\vecH^{\frac{1}{2}}_{sc}(\Gamma)} \lesssim h^{\frac{3}{2}} \|\vecg\|_{\vecL^2(D)} + h^{-\frac{1}{2}} \|\vecf\|_{\scH{1}(D)} + h^{\frac{1}{2}}\|\vecv\|_{\vecL^2(D)} + h^{-\frac{3}{2}}\|\vecu\|_{\scH{1}(D)} .
\end{eqnarray*}
Then Theorem \ref{Regularity1} yields
\begin{eqnarray} \label{SectionRegularity2Substep4vN}
|\vecv_N|_{\vecH^{\frac{1}{2}}_{sc}(\Gamma)} \lesssim h^{\frac{3}{2}} \|\vecg\|_{\vecL^2(D)} + h^{-\frac{1}{2}} \|\vecf\|_{\scH{1}(D)} .
\end{eqnarray}

{3}. (Similarly to Lemma \ref{1stRegularityuNLemma}) A priori estimate for $\vecu_N$.

\medskip

From Lemma \ref{1stRegularityuNLemma} of Section \ref{SectionRegularity1}
\begin{eqnarray*}
\mop(\lambda_2)\vecu_N&=& h^2 \mop(r_{-2}) (\mop(\rho_2)\vecv_N+\vecg_1)   \\
 &+&  h^2 \mop(r_{-1})\vecv_N  + \mop(r_{1})\vecg_2 + h\mop(r_{0}) \vecu_N .
\end{eqnarray*}
Then for small enough $h$
\begin{eqnarray*}
|\vecu_N|_{\vecH^{\frac{5}{2}}_{sc}(\Gamma)} &\lesssim& h^2 |\vecv_N|_{\vecH^{\frac{1}{2}}_{sc}(\Gamma)} +  h^2|\vecg_1|_{\vecH^{-\frac{1}{2}}_{sc}(\Gamma)} + |\vecg_2|_{\vecH^{\frac{5}{2}}_{sc}(\Gamma)} .
\end{eqnarray*}
As in estimate (\ref{Regularity1Step3estimateg2}), we need to estimate $|\gamma_N Q \nabla_h
\underline{\mdiv \vecu}|_{\vecH^{\frac{5}{2}}_{sc}(\Gamma)} $. The argument
here is different, since $\|\mdiv \vecu\|_{\scH{2}}$ can only be bounded by
$\|\vecv\|_{\scH{1}}$ from equation (\ref{Section2Step2Divu}). But $\vecv$ is
only in $\vecL^2(D)$. However,  from Lemma \ref{PreliminaryResultsDivu},
$$
|\gamma_N Q \nabla_h \underline{\mdiv \vecu}|_{\vecH^{\frac{5}{2}}_{sc}(\Gamma)} \lesssim h^{\frac{1}{2}}\|\mdiv \vecu\|_{\overline{H}_{sc}^1(D)}.
$$
Using estimate (\ref{SectionRegularity2Substep4vN}) and Theorem \ref{Regularity1}, direct calculations yield
\begin{eqnarray*}
|\vecu_N|_{\vecH^{\frac{5}{2}}_{sc}(\Gamma)} &\lesssim& h^{\frac{7}{2}} \|\vecg\|_{\vecL^2(D)} + h^{\frac{3}{2}} \|\vecf\|_{\scH{1}(D)} + h^{\frac{5}{2}}\|\vecv\|_{\vecL^2(D)} .
\end{eqnarray*}
From Theorem \ref{Regularity1} and estimate (\ref{Section6Step2Eqnu}) we now have that
\begin{eqnarray} \label{Section6Step4Eqnu} 
|\vecu_N|_{\vecH^{\frac{5}{2}}_{sc}(\Gamma)} &\lesssim& h^{\frac{7}{2}} \|\vecg\|_{\vecL^2(D)} + h^{\frac{3}{2}} \|\vecf\|_{\scH{1}(D)} .
\end{eqnarray}

{4}. New a priori estimates for $\vecv$ and $\vecu$.

\medskip

As in Section \ref{SectionRegularity1}, we have the following equation for  $\vecv$:
\begin{eqnarray*}
&&\mop(-\frac{m\rho_2(\rho_2-\rho_1+\lambda_2-\lambda_1)}{(\lambda_1-\lambda_2)(\lambda_1-\rho_2)(\rho_1-\lambda_2)(\rho_1-\rho_2)}+ \frac{m(\rho_2\lambda_2-\rho_1\lambda_1)}{(\lambda_1-\lambda_2)(\lambda_1-\rho_2)(\rho_1-\lambda_2)(\rho_1-\rho_2)}) \vecv\\
=&& -h^{-2} \left( h \gamma K_{-M} \uvecu + h^3 \gamma Q(m K_{-M}\uvecv) + h^4 \gamma Qm\tilde{Q} \underline{\vecg} + h^2 \gamma Q((1+m)\underline{\vecf}) + \frac{h}{i}\gamma Q( \nabla_h \underline{\mdiv \vecu}) \right) \\
+&& h^{-2}\mop(r_{-1}) \nabla_\Gamma (\vecu_N \cdot \nu) +  h^{-2}\mop(r_0)\vecu_N + \mop(r_{-3})\vecg_4 +h\mop(r_{-4}) \vecJ(\vecv_T) + h\mop(r_{-3})\vecv \\
:=&& \vecg_5.
\end{eqnarray*}
Then, using estimate (\ref{Section6Step4Eqnu}), we obtain
\begin{eqnarray*}
|\vecv|_{\vecH^{\frac{1}{2}}_{sc}(\Gamma)}+ |\vecJ(\vecv_T)|_{\vecH^{-\frac{1}{2}}_{sc}(\Gamma)} \lesssim h^{\frac{3}{2}} \|\vecg\|_{\vecL^2(D)} + h^{-\frac{1}{2}} \|\vecf\|_{\scH{1}(D)} + h^{\frac{1}{2}}\|\vecv\|_{\vecL^2(D)} .
\end{eqnarray*}
Therefore from equation (\ref{Section2SubStep1Eqnv}) we can obtain
\begin{eqnarray} \label{Section6Step5v}
\|\vecv\|_{\scH{1}(D)} \lesssim h^2 \|\vecg\|_{\vecL^2(D)} +  \|\vecf\|_{\scH{1}(D)} .
\end{eqnarray}
Then from equation (\ref{Section2Step1Eqnu}) we can obtain
\begin{eqnarray*}
\|\vecu\|_{\scH{3}(D)} \lesssim h^2 \|\vecv\|_{\scH{1}(D)} +  h^2\|\vecf\|_{\scH{1}(D)} + h\|\mdiv \vecu\|_{\overline{H}_{sc}^2(D)}+ h^{\frac{1}{2}} |\vecu_N|_{\vecH^{\frac{5}{2}}_{sc}(\Gamma)} .
\end{eqnarray*}
Since $\nabla m$ has compact support in $D$ and $\mdiv \left( (1+m)\vecf \right)=0$, then from equation (\ref{Section2Step2Divu}) we have that
\begin{eqnarray*}
\|\mdiv \vecu\|_{\overline{H}_{sc}^2(D)} \lesssim \|\vecu\|_{\scH{2}(D)}+ h^3 \|\vecv\|_{\scH{1}(D)} +  h^4 \|\vecg\|_{\vecL^2(D)} .
\end{eqnarray*}
Combining this inequality with (\ref{Section6Step4Eqnu}) and (\ref{Section6Step5v}) yields for small enough $h$ that
\begin{eqnarray*}
\|\vecu\|_{\scH{3}(D)} \lesssim h^2 \|\vecf\|_{\scH{1}(D)} +  h^4\|\vecg\|_{\vecL^2(D)} .
\end{eqnarray*}
We finally arrive at the following estimates
\begin{eqnarray*}
&&\|\vecv\|_{\scH{1}(D)} \lesssim h^2 \|\vecg\|_{\vecL^2(D)} +  \|\vecf\|_{\scH{1}(D)}, \\
&&\|\vecu\|_{\scH{3}(D)} \lesssim h^4\|\vecg\|_{\vecL^2(D)} + h^2\|\vecf\|_{\scH{1}(D)} .
\end{eqnarray*}

{5}. 
We use a bootstrap argument to prove the results of the theorem by repeating the above arguments line by line.
\end{proof} \proofend

\section{The inverse of $\vecB_z$} \label{InverseBz}
In this section we will show that $\vecB_z$ has a bounded inverse for some $z$ with sufficiently large $|z|$. We begin with the following. Recall that 
\begin{eqnarray*}
C(m):=\{ \arg \frac{1}{n(x)} ; \; x \in \overline D\} .
\end{eqnarray*}
Before we prove the main results in this section, we first make a connection between the set $C(m)$ and the assumptions made in Theorem \ref{Regularity1}.
\begin{lemma} \label{assumCone}
If there exists $\theta$ such that $\theta \not \in C(m) \cup \{0\} \cup \{ \arg \left( \frac{n(x)+1}{n(x)}  \right) ; \; x
\in \Gamma\}$, then $\mu=e^{i\theta}$ satisfies the assumptions in Theorem \ref{Regularity1}, i.e.
$|\xi|^2-\mu \not = 0$ , $|\xi|^2- n(x)\mu \not = 0$ for any $\xi$ and $x \in \overline{D}$ and $R_0(x,\xi')-\frac{n(x)}{1+n(x)}\mu \not =0 $ for any $\xi'$ and $x \in \Gamma$.
\end{lemma}
\begin{proof}
Assume on the contrary that there exists $ \xi \in \R^d$ such that
$$
\frac{1}{n(x)}|\xi|^2-\mu=0 \quad \mbox{or} \quad |\xi|^2-\mu=0 \quad \mbox{for some} \quad x \in  \overline{D}
$$
or 
$$
R_0(x,\xi')-\frac{n(x)}{1+n(x)}\mu =0 \quad \mbox{for some} \quad x \in  \Gamma.
$$
This implies $\theta=\mbox{arg} \,\mu \in C(m) \cup \{0\} \cup \{ \arg \left( \frac{n(x)+1}{n(x)}  \right) ; \; x
\in \Gamma\} $. This contradicts the assumption. Hence we have proved the lemma.
\end{proof} \proofend

Now we are ready to prove the following.
\begin{theorem} \label{InverseBLambda}
Assume that assumption \ref{Assumptionmconstant} holds and that $C(m) \cup \{0\} \cup \{ \arg \left( \frac{n(x)+1}{n(x)}  \right) ; \; x
\in \Gamma\} \not= [0, 2\pi[$.  Then there exists $z$ with sufficiently large
$|z|>0$ such that $\vecB_z$ has a bounded inverse $\vecR_z:
\vecF(D) \times \vecG(D) \rightarrow \vecU(D)
\times \vecV(D)$.
\end{theorem}
\begin{proof}
Since $C(m) \cup \{0\} \cup \{ \arg \left( \frac{n(x)+1}{n(x)}  \right) ; \; x
\in \Gamma\} \not= [0, 2\pi[$ , then from Lemma \ref{assumCone} there exists $\mu = e^{i\theta}$ satisfying the assumption
of Theorem \ref{Regularity1}.  Let $h >0$ and define $z:=\mu
h^{-2}$. Let $\vecB_z(\vecu, \vecv)= (\vecf,
\vecg)$ where $(\vecu, \vecv) \in \vecU(D)
\times \vecV(D)$. From Theorem \ref{Regularity1}, for a sufficiently small $h$, we have that
\begin{eqnarray} \label{InverseBLambdaEstimatev}
\|\vecv\|_{\vecL^2(D)} \lesssim |z|^{-1} \|\vecg\|_{\vecL^2(D)} + |z|^{-\frac{3}{2}} \|\mdiv \vecg\|_{\scH{1}(D)}+  \|\vecf\|_{\vecL^2(D)} + \|\mdiv \left( (1+m)\vecf \right) \|_{\scH{1}(D)} 
\end{eqnarray}
and
\begin{eqnarray}
&&\|\vecu\|_{\vecL^2(D)}+ |z|^{-\frac{1}{2}}\|\vecu\|_{\vecH^1(D)} + |z|^{-1}\|\vecu\|_{\vecH^2(D)} \nonumber \\
&\lesssim& |z|^{-1} \left(\|\vecf\|_{\vecL^2(D)} +\|\mdiv \left( (1+m)\vecf \right) \|_{\scH{1}(D)} \right) +  |z|^{-2}\|\vecg\|_{\vecL^2(D)} + |z|^{-\frac{5}{2}} \|\mdiv \vecg\|_{\scH{1}(D)}   \label {InverseBLambdaEstimateu}
\end{eqnarray}
From (\ref{RegularityDivvDivg}) (with s=1), we have that
\begin{eqnarray} \label{InverseBLambdaEstimateDivv}
\|\mdiv \vecv\|_{H^1(D)} \lesssim |z|^{-1} \|\mdiv \vecg\|_{H^1(D)}.
\end{eqnarray}
Therefore $\vecB_z$ is injective and has closed range in $\vecF(D)
\times \vecG(D)$ (the latter follows from a Cauchy sequence
argument). 

\medskip

Now we prove that $\vecB_z$ has dense range. The argument will be divided into three steps.

\medskip
 
\textit{Step 1}: First we show that for any $(\vecp^d, \vecq^d) \in \vecF(D) \times \vecG(D)$ with $\mdiv\left( \left(1+m\right) \vecp^d \right)=0$ and $\mdiv \vecq^d =0$, there exists $(\vecu_{1,\ell},\vecv_{1,\ell}) \in \vecU(D) \times  \vecV(D)$ such that
$$
\vecB_z(\vecu_{1,\ell},\vecv_{1,\ell}) \to (\vecp^d,\vecq^d) \quad \mbox{in} \quad \vecF(D) \times \vecG(D).
$$
Indeed assume that $(\vecp^d, \vecq^d) \in \vecF(D) \times \vecG(D)$ with $\mdiv\left( \left(1+m\right) \vecp^d \right)=0$ and $\mdiv \vecq^d =0$ and that
$$
\left<\vecB_z(\vecu,\vecv), (\vecp^d, \vecq^d) \right>=0, \quad \forall (\vecu,\vecv) \in \vecU(D) \times  \vecV(D)
$$
where $\left< \cdot , \cdot\right>$ denotes the natural $\vecF(D) \times \vecG(D)$
inner product. It is sufficient to show that $\vecp^d=0$ and $\vecq^d=0$ to conclude the proof in this step. As $(\vecp^d, \vecq^d)$ satisfies $\mdiv\left( \left(1+m\right) \vecp^d \right)=0$ and $\mdiv \vecq^d =0$, then the inner product reduces to the $\vecL^2$ inner product. Letting $(\vecu,\vecv) \in {\bf C}^\infty_0(D) \times {\bf
  C}^\infty_0(D)$, one gets, with $\tilde \vecp := \overline{\vecp^d}/({1+m})  $,
\begin{eqnarray*}
&\mcurl \mcurl \overline{\vecq^d}- z \overline{\vecq^d} -m \tilde \vecp= 0 \quad &\mbox{in} \quad D   \\
&\mcurl \mcurl \tilde \vecp- z(1+m) \tilde \vecp = 0 \quad &\mbox{in} \quad D  
\end{eqnarray*}
in the distributional sense. 
We observe that $\mcurl \mcurl \overline{\vecq^d} \in \vecL^2(D)$ and therefore
the tangential traces $\nu\times \mcurl \overline{\vecq^d}$ and
$\nu \times \overline{\vecq^d}$ are well defined in $\vecH^{-3/2}(\Gamma)$ and
$\vecH^{-1/2}(\Gamma)$ respectively. Since $\vecu \times \nu=0$ and $\mcurl \vecu \times \nu=0$ on $\Gamma$, then
for all $\vecv \in {\bf C}^\infty(\overline D)$ we have that
$$
\int_\Gamma (\nu\times \mcurl \overline{\vecq^d})\cdot\vecv\,ds- \int_\Gamma (\nu\times \mcurl \vecv)\cdot \overline{\vecq^d}\,ds=0
$$
where the integrals are understood as duality products. Hence 
$$
\nu \times \overline{\vecq^d} = 0 \quad \mbox{and} \quad \nu \times \mcurl \overline{\vecq^d} = 0\quad \mbox{on} \quad \Gamma
$$
(see for instance \cite[Lemma 3.1]{H}). Now Let $\vecp_1=z\overline{\vecq^d} - \tilde \vecp$. Then one gets
\begin{eqnarray}
&\mcurl \mcurl \overline{\vecq^d}- z(1+m) \overline{\vecq^d} + m \vecp_1= 0 \quad &\mbox{in} \quad D   \label{InverseBLambdaEqnq}\\
&\mcurl \mcurl \vecp_1- z \vecp_1 = 0 \quad &\mbox{in} \quad D . \label {InverseBLambdaEqnp1}
\end{eqnarray}
Now we want to apply Theorem \ref{Regularity1} (one can check that we can relax the condition $\mcurl \vecq^d \in \vecL^2(D)$ from the proof of Theorem \ref{Regularity1}) to
$(\overline{\vecq^d}, \vecp_1)$. Since $z=\mu h^{-2}$ and $\mu$ satisfies the assumption in Theorem \ref{Regularity1}, we obtain  $\overline{\vecq^d} =0,\,
\vecp_1=0$ which implies $\vecp^d=0, \, \vecq^d=0$. This proves the first part.

\medskip

\textit{Step 2}: We show that for any given $(\vecp^c, \vecq^c) \in \vecF(D) \times \vecG(D)$ with $\mcurl \vecp^c =0$, $\vecp^c \times \nu |_\Gamma=0$ and $\mcurl \vecq^c =0$, $\vecq^c \times \nu |_\Gamma=0$, there exists $(\vecu_{2,\ell}, \vecv_{2,\ell})$ such that
$$
\vecB_z(\vecu_{2,\ell},\vecv_{2,\ell})  \to (\vecp^c,\vecq^c) \quad \mbox{in} \quad \vecF(D) \times \vecG(D).
$$
Assume 
$$
\left<\vecB_z(\vecu,\vecv), (\vecp^c, \vecq^c) \right>=0, \quad \forall (\vecu,\vecv) \in \vecU(D) \times  \vecV(D).
$$
It is sufficient to show $\vecp^c=0$ and $\vecq^c=0$ to conclude the proof in this step. Indeed  from $\mcurl \vecp^c=0$, $\mdiv
\left( (1+m)\vecp^c\right) \in H^1(D)$ and $ \vecp^c \times \nu|_\Gamma=0$, one
gets $\vecp^c \in \vecH^2(D)$ (see \cite{ABDG}), then $\mcurl \vecp^c=0$
implies $\vecp^c \in \vecU(D)$. We obviously have
$\vecq^c \in \vecV(D)$. Then, letting $\vecu=\vecp^c$ and $\vecv=0$, one gets 
$$
\|\vecp^c\|_{\vecF(D)} =0 .
$$
This implies $\vecp^c=0$. Second, let $\vecv=\vecq^c$ which implies
$$
\|\vecq^c\|_{\vecG(D)} =0
$$
and therefore $\vecq^c=0$. 

\medskip

\textit{Step 3}: Now we are ready to prove that $\vecB_z$ has dense range in $\vecF(D) \times \vecG(D)$. Indeed let $(\vecp, \vecq) \in \vecF(D) \times \vecG(D)$. By the Helmholtz decomposition (see for instance
\cite{KH}), there exist unique  $\vecp^d \in \vecL^2(D)$, $\vecp^c \in \vecL^2(D)$ and $\vecq^d \in \vecL^2(D)$, $\vecq^c \in \vecL^2(D)$ such that
\begin{equation} 
\vecp=\vecp^d + \vecp^c, \quad \vecq=\vecq^d + \vecq^c
\end{equation}
where
\begin{eqnarray*}
\mdiv \left( \left( 1+m \right)\vecp^d \right) =0, \quad \mcurl \vecp^c =0, \quad \vecp^c \times \nu|_\Gamma=0.
\end{eqnarray*}
\begin{eqnarray*}
\mdiv \vecq^d=0, \quad \mcurl \vecq^c =0, \quad \vecq^c \times \nu|_\Gamma=0.
\end{eqnarray*}
The existence of $(\vecp^d, \vecp^c)$ is guaranteed by the strict positiveness of $\Re(1+m)$. As shown above, there exists $(\vecu_{1,\ell},\vecv_{1,\ell}) \in \vecU(D) \times  \vecV(D)$ and $(\vecu_{2,\ell}, \vecv_{2,\ell})  \in \vecU(D) \times  \vecV(D) $ such that
$$
\vecB_z(\vecu_{1,\ell},\vecv_{1,\ell}) \to (\vecp^d,\vecq^d) \quad \mbox{in} \quad \vecF(D) \times \vecG(D)
$$
and
$$
\vecB_z(\vecu_{2,\ell},\vecv_{2,\ell})  \to (\vecp^c,\vecq^c) \quad \mbox{in} \quad \vecF(D) \times \vecG(D).
$$
Now let $\vecu_\ell= \vecu_{1,\ell} + \vecu_{2,\ell}$ and $\vecv_\ell= \vecv_{1,\ell}  +\vecv_{2,\ell}$. Then
$$
\vecB_z(\vecu_\ell,\vecv_\ell) \to (\vecp^d+ \vecp^c, \vecq^d+ \vecq^c) = (\vecp,\vecq)
$$
in $\vecF(D) \times \vecG(D)$. Now we have proved that $\vecB_z$ has dense range in $\vecF(D) \times \vecG(D)$. Since $\vecB_z$ is injective and has closed dense range in $\vecF(D) \times \vecG(D)$, $\vecR_z:=\vecB_z^{-1}$ is well-defined. 
\end{proof} \proofend 
\section{Main results on transmission eigenvalues} \label{Main}
We shall state and prove here the main results of our paper on the existence of
transmission eigenvalues and the completeness  of associated eigenvectors. The
results of this section heavily rely on the regularity results obtained in section \ref{Regularity}.

Let us first introduce the Helmholtz decomposition. The motivation for introducing Helmholtz decomposition is to get the desired compact imbedding (which will be proved to be a Hilbert-Schmidt operator) for Maxwell's equations.
 For any $\vecu \in
\vecL^2(D)$ there exists a unique  $\vecu^d \in \vecL^2(D)$ and $\vecu^c \in \vecL^2(D) $ such that
\begin{equation} \label{decomposition}
\vecu=\vecu^d + \vecu^c
\end{equation}
and
\begin{eqnarray*}
\mdiv \left( \left( 1+m \right)\vecu^d \right) =0, \quad \mcurl \vecu^c =0, \quad \vecu^c \times \nu|_\Gamma=0.
\end{eqnarray*}
This is guaranteed by the strict positiveness of $\Re(1+m)$ (see for instance
\cite{KH}). We now define $\vecP^d  $ as the projection operator in $
\vecL^2(D) \times \vecL^2(D)$ defined by
$$\vecP^d (\vecu,\vecv)= (\vecu^d,\vecv)$$
where $\vecu^d$ is defined by \eqref{decomposition}. 

\medskip

For $z$ chosen as in Theorem \ref{InverseBLambda}, we now consider the operator
$$\vecS_z:=\vecP^d \vecR_z : \vecH(D) \to \vecH(D)$$ 
with 
$$\vecH(D):=\{\vecu \in \vecU(D); \mdiv \left( \left(1+m\right)\vecu\right)=0\}  \times \{\vecv \in \vecL^2(D); \mdiv  \vecv =0\} .
$$

Since $\vecH(D)$ is a subspace
of $\vecH^2(D) \times \vecG$, we also get from  Theorem \ref{Regularity2} that $\vecS_z^2$
continuously map $\vecH(D)$ into $\vecH^6(D) \times \vecH^4(D)$.  
Observing that the $\vecH^2(D)\times \vecL^2(D)$ norm is an equivalent norm in
$\vecH(D)$, we have from \cite[Lemma 4.1]{R1} (see also \cite{A}) that  $\vecS_z^2 : \vecH(D) \to
\vecH(D)$  is a Hilbert-Schmidt operator.

\medskip

We shall now apply Agmon's theory on the spectrum of Hilbert-Schmidt
operators in \cite{A} to get the desired main results. More
specifically we shall apply the result of the following lemma that is a direct
consequence of Proposition 4.2 and the proof of Theorem 5 in \cite{R1}. 

\begin{lemma} \label{SpectralLemma}
Let $H$ be a Hilbert space and $S$ be a bounded linear operator from $H$ to $H$.  If $\lambda^{-1}$ is in the resolvent of $S$, define 
$$
(S)_\lambda=S(I-\lambda S)^{-1}.
$$
Assume $S^p : H \to H$ is a Hilbert-Schmidt operator for some $p \ge 2$. For
the operator $S$, assume there exists $N$ rays with bounded growth where the
angle between any two adjacent rays is less that $\frac{\pi}{2p}$: more
precisely 
assume there exist $0\le \theta_1 < \theta_2 <\cdots <\theta_N < 2\pi$ such that $\theta_k-\theta_{k-1}<\frac{\pi}{2p}$ for $k=2,\cdots,N$ and $2\pi-\theta_N+\theta_1<\frac{\pi}{2p}$ satisfying the condition that there exists $r_0>0$, $c>0$ such that $sup_{r\ge r_0}\|(S)_{re^{i\theta_k}}\| \le c$ for $k=1,\cdots,N$. Then the space spanned by the nonzero generalized eigenfunctions of $S$ is dense in the closure of the range of $S^p$. 
\end{lemma} 
We shall first apply this lemma to the operator $\vecS_z$, then deduce the
spectral decomposition of the operator $\vecB_z$ and the main result on
transmission eigenvalues. In order to prove the existence of rays with bounded
growth  we need the following two lemmas on $(\vecR_z)_\lambda$ which will be used in the proof of Theorem \ref{TheoremPdRLambdaPd}.

\begin{lemma} \label{InverseSz} Let $z\in \C$ such $\vecR_z=\vecB_z^{-1}$ is
  well defined as in Theorem \ref{InverseBLambda}.   Then one has the following identities:
$$
\vecP^d \vecR_z \vecP^d \vecB_z = \mathbf{I}, \quad \mbox{and} \quad \vecP^d \vecB_z \vecP^d \vecR_z = \mathbf{I}
$$
where $\mathbf{I}$ is the identity operator on $\vecH(D)$. 
\end{lemma}
\begin{proof}
On one hand, for any $(\vecf^d, \vecg) \in \vecH(D)$, let $(\vecu, \vecv)=\vecR_z (\vecf^d, \vecg)$, then
\begin{eqnarray*}
&\mcurl \mcurl \vecu-z (1+m) \vecu - m \vecv= (1+m)\vecf^d \quad &\mbox{in} \quad D  \\
&\mcurl \mcurl \vecv-z \vecv = \vecg \quad &\mbox{in} \quad D 
\end{eqnarray*}
Let $(\vecu^d,\vecv)=\vecP^d(\vecu,\vecv)$, then
\begin{eqnarray*}
&\mcurl \mcurl \vecu^d-z(1+m) \vecu^d - m \vecv= (1+m)\vecf^d + z (1+m) \vecu^c \quad &\mbox{in} \quad D  \\
&\mcurl \mcurl \vecv-z \vecv = \vecg \quad &\mbox{in} \quad D 
\end{eqnarray*}
This implies that
$$
\vecP^d \vecB_z \vecP^d \vecR_z (\vecf^d, \vecg)  = \vecP^d \vecB_z \vecP^d (\vecu,\vecv) =\vecP^d \vecB_z  (\vecu^d,\vecv) = \vecP^d (\vecf^d + z \vecu^c,\vecg)=(\vecf^d, \vecg).
$$

On the other hand, for any $(\vecu^d, \vecv) \in \vecH(D)$, let $(\vecf, \vecg)=\vecB_z (\vecu^d, \vecv)$, then
\begin{eqnarray*}
&\mcurl \mcurl \vecu^d-z (1+m) \vecu^d - m \vecv= (1+m)\vecf \quad &\mbox{in} \quad D  \\
&\mcurl \mcurl \vecv-z \vecv = \vecg \quad &\mbox{in} \quad D 
\end{eqnarray*}
This implies that
\begin{eqnarray*}
&\mcurl \mcurl (\vecu^d + \frac{1}{z} \vecf^c)-z (1+m) (\vecu^d+ \frac{1}{z} \vecf^c) - m \vecv= (1+m)\vecf^d \quad &\mbox{in} \quad D  \\
&\mcurl \mcurl \vecv-z \vecv = \vecg \quad &\mbox{in} \quad D 
\end{eqnarray*}
Therefore
$$
\vecP^d \vecR_z \vecP^d \vecB_z (\vecu^d, \vecv) = \vecP^d \vecR_z (\vecf^d, \vecg) = \vecP^d (\vecu^d + \frac{1}{z} \vecf^c, \vecv) = (\vecu^d, \vecv).
$$
Hence we have proved the lemma.
\end{proof} \proofend

We now have the following expression for $(\vecS_z)_\lambda$.

\begin{lemma} \label{SzLambda} 
Let $\lambda \in \C$ and 
assume that $\vecR_{z+\lambda} = \vecB_{z+\lambda}^{-1} $ is well defined. Then $(\vecS_z)_\lambda = \vecP^d \vecR_{z+\lambda}$.
\end{lemma}
\begin{proof}
By definition, $(\vecS_z)_\lambda = \vecP^d \vecR_z (\mathbf{I}-\lambda \vecP^d
\vecR_z)^{-1}$. From Lemma \ref{InverseSz} and the fact that $\vecP^d
\mathbf{I} = \mathbf{I} $ where $\mathbf{I}$ is the identity operator on
$\vecH(D)$, we have that
\begin{eqnarray*}
(\vecS_z)_\lambda 
&=& \vecP^d \vecR_z (\mathbf{I}-\lambda \vecP^d \vecR_z)^{-1} \\
&=& \vecP^d \vecR_z (\vecP^d \vecB_z \vecP^d \vecR_z - \lambda  \vecP^d \vecR_z)^{-1} \\
&=& \vecP^d \vecR_z ((\vecP^d \vecB_z - \lambda \mathbf{I}) \vecP^d \vecR_z  )^{-1} \\
&=& \vecP^d \vecR_z (\vecP^d (\vecB_z - \lambda \mathbf{I}) \vecP^d \vecR_z  )^{-1} \\
&=& \vecP^d \vecR_z (\vecP^d \vecB_{z+\lambda}  \vecP^d \vecR_z  )^{-1} \\
&=&  \vecP^d \vecR_{z+\lambda}.
\end{eqnarray*}
where for the last equality we used that $\vecP^d \vecR_{z+\lambda}\vecP^d
\vecB_{z+\lambda} = \mathbf{I}$.
\end{proof} \proofend

We are now in position to prove the following result on the spectral
decomposition of  $\vecS_z$.
\begin{theorem} \label{TheoremPdRLambdaPd}
Assume that Assumption 1 holds and assume that $C(m)$ is contained in an interval of
length $< \frac{\pi}{4}$.  Then there are infinitely many eigenvalues of
$\vecS_z$ and the associated generalized eigenfunctions are dense in $\{\vecu \in \vecU(D); \mdiv \left( \left(1+m\right)\vecu\right)=0\} \times \{\vecv \in \vecV(D); \mdiv  \vecv =0\}$.
\end{theorem}
\begin{proof} We prove the theorem in two steps.

\medskip

\textit{Step 1}. We shall apply Lemma \ref{SpectralLemma} with $S=\vecS_z$, $H=\vecH(D)$ and $p=2$. Since $C(m)$ is contained  in an interval of
length $< \frac{\pi}{4}$, then we can choose $0\le \theta_1 < \theta_2 <\cdots <\theta_N < 2\pi$ such that (recall that since $n$ is a constant on $\Gamma$, then $\{ \arg \left( \frac{n(x)+1}{n(x)}  \right) ; \; x
\in \Gamma\}$ is a fixed angle)
$$
\theta_k-\theta_{k-1}<\frac{\pi}{4}
$$ for $k=2,\cdots,N$ and $2\pi-\theta_N+\theta_1<\frac{\pi}{4}$ 
satisfying
$$
\theta_j \not \in C(m)\cup \{0\} \cup \{ \arg \left( \frac{n(x)+1}{n(x)}  \right) ; \; x
\in \Gamma\} .
$$
From Lemma \ref{assumCone} and Theorem \ref{InverseBLambda}, $\vecR_{re^{i\theta_k}}$ is well-defined as the bounded inverse of $\vecB_{re^{i\theta_k}}$. Moreover $\vecR_{re^{i\theta_k}}$ is uniformly bounded with respect to $r$ because of the estimates (\ref{InverseBLambdaEstimatev}), (\ref{InverseBLambdaEstimateu}) and (\ref{InverseBLambdaEstimateDivv}). Now for sufficiently large $r>0$, the angle of $z+re^{i\theta_k}$ is sufficiently close to $re^{i\theta_k}$. Therefore $\vecR_{z+re^{i\theta_k}}$ is also uniformly bounded with respect to $r$. Hence there exist $r_0$ such that 
$$
sup_{r\ge r_0}\|\vecR_{z+re^{i\theta_k}}\| \le c .
$$
From Lemma \ref{SzLambda} we have that
$$
S_{re^{i\theta_k}} = (\vecS_z)_{re^{i\theta_k}}=\vecP^d \vecR_{z+re^{i\theta_k}}.
$$
Therefore
$$
sup_{r\ge r_0}\|S_{re^{i\theta_k}}\| \le c .
$$

Now we have found directions $\theta_j$ as required in Lemma \ref{SpectralLemma}
for which the bounded growth conditions are satisfied. 

\medskip

\textit{Step 2}. It only remains to prove that the closure of the
range of  $\vecS_z^2$ is dense in  $\{\vecu \in \vecU: \mdiv \left( \left(1+m\right)\vecu\right)=0\} \times \{\vecv \in \vecV: \mdiv  \vecv =0\}$. By a denseness argument, it is sufficient to show that the
closure of the range of  $\vecS_z$ is $\{\vecu \in \vecU(D): \mdiv \left( \left(1+m\right)\vecu\right)=0\} \times \{\vecv \in \vecV(D): \mdiv  \vecv =0\}$. 
Indeed for  $(\vecu,\vecv) \in \{\vecu \in \vecU(D): \mdiv \left( \left(1+m\right)\vecu\right)=0\} \times \{\vecv \in \vecV(D): \mdiv  \vecv =0\}$, we define $p \in H^1_0(D)$ such that 
\begin{eqnarray*}
-z \mdiv [(1+m) \nabla p] = \nabla m \cdot \vecv 
\end{eqnarray*}
Since $\mcurl \nabla p=0$, $\mdiv \nabla p \in L^2(D)$ and $\nu \times \nabla
p=0$ then $\nabla p \in \vecH^1(D)$ (see for instance \cite{ABDG}), the same
argument yields again $\nabla p \in \vecH^2(D)$  since $\mdiv [(1+m) \nabla p]
\in H^1(D)$ (this come from the fact that $\nabla m$ has compact support in $D$
and $\vecv$ is regular on that support by elliptic regularity).  

\medskip

Let $\vecu^*=\vecu+\nabla p$. Then we have $(\vecu^*, \vecv)\in
\vecU(D) \times \vecV(D)$ and $\vecP^d(\vecu^*,\vecv)=(\vecu,\vecv)$. Moreover by a direct calculation we have that
$$
\mdiv(-z(1+m)\vecu^*-m\vecv)=0.
$$
Now define $(\vecf,\vecg)= \vecB_z (\vecu^*,\vecv)$. Then 
$$
(\vecf,\vecg)  \in \{\vecf \in \vecF(D); \mdiv \left( \left(1+m\right)\vecf\right)=0\} \times \{\vecg \in \vecG(D); \mdiv  \vecg =0\} .
$$
Let $(\vecf_\ell,\vecg_\ell) \in  \vecF(D) \times\vecG(D)$ be a Cauchy sequence such that 
$$
(\vecf_\ell,\vecg_\ell) \to (\vecf,\vecg)
$$ 
in the space $\vecF(D) \times\vecG(D)$. Since $\vecR_z$ is bounded, we have that
\begin{eqnarray*}
\vecR_z(\vecf_\ell,\vecg_\ell)\to \vecR_z(\vecf,\vecg) = (\vecu^*,\vecv)  \quad \mbox{in} \quad \vecU(D) \times \vecV(D).
\end{eqnarray*}
Therefore 
$$
\vecS_z (\vecf_\ell,\vecg_\ell) =\vecP^d \vecR_z (\vecf_\ell,\vecg_\ell) \to \vecP^d(\vecu^*,\vecv)=(\vecu,\vecv)
$$
in $\{\vecu \in \vecU(D); \mdiv \left( \left(1+m\right)\vecu\right)=0\} \times \{\vecv \in \vecV(D); \mdiv  \vecv =0\}$. This proves the theorem.

\end{proof} \proofend 

Now we relate the transmission eigenvalues to the operator $\vecB_z$.
\begin{theorem} \label{FormulationITEBk}
The number $k$ and $(\vecu, \vecv) \in \vecU(D)
\times \{\vecv \in \vecV(D): \mdiv  \vecv =0\}$  are a transmission eigenvalue
and a non trivial solution of (\ref{IntroEu})-(\ref{IntroEv}) respectively if and only if
$\mu^{-1} = k^2 - z$ and $ \vecP^d(\vecu, \vecv)$ are respectively  an eigenvalue and an
eigenvector of $\vecS_z$.
\end{theorem}
\begin{proof}
First we show that for each eigenvalue $\mu^{-1}$ of $\vecS_z$ we can find a transmission eigenvalue $k$ and non trivial solution of (\ref{IntroEu})-(\ref{IntroEv}). Indeed, suppose $(\vecu^d,\vecv) \in \vecH(D)$ is such that
\begin{eqnarray} \label{FormulationPdEigenfunction}
\vecP^d \vecB^{-1}_z  (\vecu^d,\vecv) = \mu^{-1} (\vecu^d,\vecv) .
\end{eqnarray}
Since $\vecB^{-1}_z$ is well-defined, $(\tilde{\vecu},\tilde{\vecv}):=\mu \vecB_z^{-1} (\vecu^d,\vecv) $ satisfies
\begin{eqnarray*}
&\mcurl \mcurl \tilde{\vecu}- z (1+m) \tilde{\vecu}- m \vecv= \mu (1+m)\vecu^d \quad &\mbox{in} \quad D  \label{FormulationTildeEu}  \\
&\mcurl \mcurl \tilde{\vecv}- z \tilde{\vecv} = \mu \vecv \quad &\mbox{in} \quad D.
\end{eqnarray*}
Define $\tilde{\vecu}^d$ such that $(\tilde{\vecu}^d,\tilde\vecv) = \vecP^d(
\tilde{\vecu},\tilde\vecv) $. Then, equation (\ref{FormulationPdEigenfunction}) yields 
$$
\tilde{\vecu}^d = \vecu^d, \quad \quad \tilde{\vecv}=\vecv .
$$
Now set 
\begin{eqnarray*}
\vecu= \tilde{\vecu}^d + \frac{z}{\mu+z}\tilde{\vecu}^c= {\vecu}^d +
\frac{z}{\mu+z}\tilde{\vecu}^c, 
\end{eqnarray*}
where $\tilde{\vecu}^c = \tilde{\vecu} -  \tilde{\vecu}^d$. Then a direct calculation yields
\begin{eqnarray*}
&\mcurl \mcurl \vecu- (z+\mu) (1+m) \vecu- m \vecv= 0 \quad &\mbox{in} \quad D  \ \\
&\mcurl \mcurl \vecv- (z+\mu) \vecv = 0 \quad &\mbox{in} \quad D .
\end{eqnarray*}
The definition of $\vecu$ and \eqref{decomposition} ensures that
$\gamma_t\vecu=0$ and $\gamma_t \mcurl \vecu=0$ on $\Gamma$ and that
$(\vecu,\vecv)$ are non trivial solutions of
(\ref{IntroEu})-(\ref{IntroEv}) with $k:=\sqrt{z+\mu}$ (with appropriate branch). 

The converse is easily seen by reversing the above arguments and defining
$(\vecu^d,\vecv)= \vecP^d(\vecu, \vecv)$. This completes the proof.
\end{proof} \proofend \\

{\bf \noindent Proof of Theorem \ref{mainHM}: }Note that since $\vecS_z^2$ is a Hilbert-Schmidt operator then the reciprocal
of the eigenvalues form a discrete set without finite accumulation points. We
therefore can summarize the results on transmission eigenvalues in Theorem \ref{mainHM}. \proofend
\subsection{Discussion} The assumption  that the refraction index $n$ is constant near the
boundary (Assumption \ref{Assumptionmconstant}) is in fact only needed  in
Section \ref{Regularity} to establish desired regularity results. The
arguments of Section \ref{InverseBz} and Section \ref{Main} are still valid for non constant
$n$ if the regularity result holds. Relaxing Assumption \ref{Assumptionmconstant} is part of an ongoing
project where we think that the (simpler) approach in \cite{S} would be feasible.

\section{Appendix}
\appendix
We introduce a small parameter $h$. We define $D^h_{x_j}=\frac{h}{i}\frac{\partial}{\partial x_j}$. Similar notations hold for $\nabla_h, \frac{\partial_h}{\partial \nu}$. For an open bounded manifold $D$ in $\R^3$ we introduce the semiclassical Sobolev spaces $\scH{s}(D)$ equipped with the norm $\|\cdot\|_{\scH{s}(D)}$, where $\|\vecu\|_{\scH{s}(D)}:= \mbox{inf}\{ \|\tilde{\vecu}\|_{\vecH_{sc}^s(\R^3)}, \tilde{\vecu}|_D = \vecu\}$ and $\|\vecu\|_{\vecH_{sc}^s(\R^3)}^2:= \int_{\R^3} (1+h^2|\xi|^2)^{s} |\hat{\vecu}(\xi)|^2 d\xi$. For a two dimensional manifold $\Gamma$, we denote the semiclassical norm as $|\cdot|_{\vecH_{sc}^{s}(\Gamma)}$. We denote the commutator of two semiclassical pseudo-differential operators as $[\cdot,\cdot]$. We refer to \cite{AG} and \cite{Z} for details. By $a \lesssim b$ we mean that $a \le Cb$ for some independent constant $C$.

\medskip
\begin{definition}
Let $a(x,\xi)$ be in $C^\infty(\R^{2d})$, we say $a$ is a symbol of order $m$, denoted as $a \in S^m$, if 
$$
|\partial_x^\alpha \partial_\xi^\beta a(x,\xi)| \le C_{\alpha \beta} \langle \xi \rangle ^{m-|\beta|}
$$
for all $\alpha$ and $\beta$ where $\langle \xi \rangle:=(1+|\xi|^2)^{\frac{1}{2}}$. For $a\in S^m$ we define the semiclassical operator $\mOp_h(a)$ by
$$
\mOp_h(a) u = \frac{1}{(2\pi)^d} \int e^{i x \xi} a(x,h\xi) \hat{u}(\xi) d\xi
$$
and we define the class of such operators as $\mOp_h S^m$.
\end{definition}

In particular we need the following results from \cite{R1}. Let $x=(x',x_n)$ and $\xi=(\xi',\xi_n)$ where $(x,\xi)$ is the local coordinate in the cotangent bundle $T^*(\Gamma \times (0,\epsilon))$ and $(x',\xi')$ is the local coordinate in the cotangent bundle $T^*\Gamma$. 

\medskip

For the case that $-h^2\Delta -\mu$ is elliptic with the symbol $|\xi|^2-\mu \not=0$ for any $\xi$ and $x \in \overline{D}$, we have in the tubular neighborhood of $\Gamma$ the semiclassical symbol of 
$$
-h^2\Delta -\mu
$$ 
is 
$$
\xi_n^2 +2hH(x)\frac{1}{i}\xi_n+ R(x,\xi')-\mu
$$
where $H(x)$ is a smooth function depending on $x$. We denote by 
\begin{eqnarray} \label{PreliminaryResultsPrincipleSemiclassicalSymbol}
R_0(x,\xi') 
\end{eqnarray}
the principal semiclassical symbol of $R(x,\xi')$. Moreover we can have
\begin{eqnarray} \label{PreliminaryResultsPolynomial1}
\xi_n^2 + R(x,\xi')-\mu = (\xi_n - \rho_1(x,\xi'))(\xi_n-\rho_2(x,\xi'))
\end{eqnarray}
where $\rho_1$ and $\rho_2$ are symbols of order 1 with $\Im(\rho_1)>0$ and $\Im(\rho_2)<0$.

For the case that $-h^2\Delta -\mu(1+m)$ is elliptic with the symbol $|\xi|^2-\mu(1+m) \not=0$ for any $\xi$ and $x \in \overline{D}$ we have similarly
\begin{eqnarray} \label{PreliminaryResultsPolynomial2}
\xi_n^2 + R(x,\xi')-\mu(1+m) = (\xi_n - \lambda_1(x,\xi'))(\xi_n-\lambda_2(x,\xi'))
\end{eqnarray}
where $\lambda_1$ and $\lambda_2$ are symbols of order 1 with $\Im(\lambda_1)>0$ and $\Im(\lambda_2)<0$. 

\medskip

Also we will use frequently that if the symbol $|\xi|^2-\mu(1+m) \not= 0$ for all $\xi$ and $x \in \overline{D}$, then the parametrix $Q$ of $-h^2 \Delta -\mu(1+m)$ exists
where
$$
Q \left( -h^2 \Delta -\mu(1+m) \right)=I 
$$
modulo a smoothing operator. The following holds
\begin{eqnarray} \label{PreliminaryResultsEstimateQ}
\|(Q\underline{\vecf})|_D\|_{\scH{s+2}(D)} \lesssim \|\vecf\|_{\scH{s}(D)}
\end{eqnarray}
for any $\vecf \in \scH{s}(D)$ with $s \ge 0$. The same holds true for the parametrix $\tilde{Q}$ of $-h^2 \Delta -\mu$. Also we have
\begin{eqnarray} \label{PreliminaryResultsBounaryToDomain}
\|\left( Q ( {\bf \psi} \otimes (D_s^h)^k \delta_{s=0}) \right) |_D \|_{\scH{s-k+\frac{3}{2}}(D)} \lesssim h^{-\frac{1}{2}} \|{\bf \phi}\|_{\vecH_{sc}^{s}(\Gamma)} 
\end{eqnarray}
where $s-k+\frac{3}{2} \ge 0$.

Moreover if $-h^2 \Delta \vecv - \mu \vecv=h^2\vecg$ in $D$ and $|\xi|^2-\mu\not=0$ then 
\begin{eqnarray} \label{PreliminaryResultsvBoundary}
\| \vecv \|_{\scH{s+1}(D \backslash \overline{N})} \lesssim h\|\vecv\|_{\scH{s}(D)} + h^2\|\vecg\|_{\scH{max\{s-1,0\}}(D)}
\end{eqnarray} 
for $s\ge0$ when the right hand side makes sense.

\medskip

Next we introduce $\mbox{op}(r_M)$ as the semiclassical pseudo-differential operator of order $M$ on $\Gamma$. We have that
\begin{eqnarray} \label{PreliminaryResultsTildeQBoundary1}
\gamma \tilde{Q} [\frac{h}{i}{\bf \psi} \otimes D^h_s \delta_{s=0}]=\mop(\frac{\rho_1}{\rho_1-\rho_2}) {\bf \psi}+ h \mop(r_{-1}) {\bf \psi}
\end{eqnarray}
\begin{eqnarray} \label{PreliminaryResultsTildeQBoundary2}
\gamma \tilde{Q} [\frac{h}{i} {\bf \psi} \otimes \delta_{s=0}]=\mbox{op}(\frac{1}{\rho_1-\rho_2}) {\bf \psi}+ h \mbox{op}(r_{-2}) {\bf \psi}
\end{eqnarray}
\begin{eqnarray} \label{PreliminaryResultsQmTildeQBoundary1}
\gamma Q m \tilde{Q} [\frac{h}{i} {\bf \psi} \otimes \delta_{s=0}]=\mop\left( \frac{m(\rho_2-\rho_1+\lambda_2-\lambda_1)}{(\lambda_1-\lambda_2)(\lambda_1-\rho_2)(\rho_1-\lambda_2)(\rho_1-\rho_2)} \right) {\bf \psi} + h \mop(r_{-4}) {\bf \psi}
\end{eqnarray}
\begin{eqnarray} \label{PreliminaryResultsQmTildeQBoundary2}
\gamma Qm\tilde{Q} [\frac{h}{i} {\bf \psi} \otimes D^h_s \delta_{s=0}]=\mop \left( \frac{m(\rho_2 \lambda_2- \rho_1 \lambda_1)}{(\lambda_1-\lambda_2)(\lambda_1-\rho_2)(\rho_1-\lambda_2)(\rho_1-\rho_2)} \right) {\bf \psi}+ h \mop(r_{-3}) {\bf \psi}
\end{eqnarray}
where ${\bf \psi}$ is a distribution on the boundary.

\medskip

In the framework of semiclassical norms, the trace formula reads
\begin{eqnarray} \label{PrelinimaryResultsTrace}
|\gamma u|_{H_{sc}^{s-\frac{1}{2}}(\Gamma)} \lesssim h^{-\frac{1}{2}} \|u\|_{\overline{H}_{sc}^s(D)}
\end{eqnarray}
for $s>\frac{1}{2}$.

\medskip

Moreover we need the following two lemmas.
\begin{lemma} \label{PreliminaryResultsQNablau}
Assume $u \in H^s(D)$. Then for $s \ge 0$
\begin{eqnarray*}
\|Q \nabla_h \underline{u} \|_{\overline{\vecH}_{sc}^{s+1}(D)} \lesssim \|u\|_{\overline{H}_{sc}^s(D)}.
\end{eqnarray*}
\end{lemma}
\begin{proof}
If $s=0$, then this is a consequence of the mapping properties of semiclassical pseudo-differential operators on $L^2(\R^d)$. Now assume $s \ge 1$. From classical jump relations (c.f. \cite{N})
\begin{eqnarray*}
\nabla_h \underline{u} &=& \underline{\nabla_h u} + (\nu \frac{h}{i}u) \otimes \delta_{s=0}.
\end{eqnarray*}
Then 
\begin{eqnarray*}
\|Q \nabla_h \underline{u} \|_{\overline{\vecH}_{sc}^{s+1}(D)} \lesssim \|Q \underline{\nabla_h u} \|_{\overline{H}_{sc}^{s+1}(D)} + \|Q(\frac{h}{i} \nu u \otimes \delta_{s=0})\|_{\overline{H}_{sc}^{s+1}(D)}.
\end{eqnarray*}
From the estimates (\ref{PreliminaryResultsBounaryToDomain}) and (\ref{PrelinimaryResultsTrace}) we have that
\begin{eqnarray*}
\|Q(\frac{h}{i} \nu u \otimes \delta_{s=0})\|_{\overline{H}_{sc}^{s+1}(D)} \lesssim \|u\|_{\overline{H}_{sc}^s(D)}.
\end{eqnarray*}
Noting that $s \ge 1$, we can proceed to have
\begin{eqnarray*}
\|Q \underline{\nabla_h u} \|_{\overline{H}_{sc}^{s+1}(D)} \lesssim \|\nabla_h u\|_{\overline{H}_{sc}^{s-1}(D)} \lesssim \|u\|_{\overline{H}_{sc}^s(D)}.
\end{eqnarray*}
This completes our proof.
\end{proof} \proofend

\begin{lemma} \label{PreliminaryResultsDivu}
Assume $f \in H^1(D)$ and $f=0$ in the neighborhood $N$ of the boundary $\Gamma$. Then for $f \in \overline{H}_{sc}^s(D)$ and small enough $h$
\begin{eqnarray}
|\gamma Q \nabla_h f |_{\vecH^{s+\frac{3}{2}}(\Gamma)} \lesssim h^{\frac{1}{2}}\|f\|_{\overline{H}_{sc}^s(D)}.
\end{eqnarray}
\end{lemma}
\begin{proof}
Note that if $f =0$ in $N$, then $\underline{f} \in H_{sc}^s(\R^3)$ and $\underline{f} \in H^1(\R^3)$. Let $\vecu \in \vecH_{sc}^s(\R^3)$ satisfy 
\begin{eqnarray*}
-h^2\Delta \vecu - \mu (1+m) \vecu = \nabla_h \underline{f}.
\end{eqnarray*}
Then $ \vecu =  Q \nabla_h \underline{f}+hK_{-M}\vecu$ for sufficiently large $M>0$.  Let $\chi \in C_0^\infty (\R^3)$ be supported in $N_\epsilon=\left\{ x: x= y+ s \nu(y), \, y \in \Gamma, \, -\epsilon \le s < \epsilon \right\}$ with sufficiently small $\epsilon>0$ such that $\chi \nabla_h f =0$, and $\chi=1$ on $\Gamma$. Then we have
\begin{eqnarray} \label{PreliminaryResultsDivuEqn1}
\| \chi Q \nabla_h \underline{f} \|_{\vecH_{sc}^{s+2}(\R^3)} \le \| \chi \vecu \|_{\vecH_{sc}^{s+2}(\R^3)} + h\| \vecu \|_{\vecH_{sc}^{s+1}(\R^3)} .
\end{eqnarray}
Since $\chi \nabla_h f=0$ then
\begin{eqnarray*}
-h^2\Delta (\chi \vecu) - \mu (1+m) \chi \vecu = \chi \nabla_h \underline{f} - h K_{1} \vecu = -hK_{1} \vecu
\end{eqnarray*}
where $K_1$ is a differential operator of order $1$. Therefore 
\begin{eqnarray*}
\|\chi \vecu\|_{\vecH_{sc}^{s+2}(\R^3)} \lesssim h \|\vecu\|_{\vecH_{sc}^{s+1}(\R^3)}.
\end{eqnarray*}
Then estimate (\ref{PreliminaryResultsDivuEqn1}) yields
\begin{eqnarray*}
\|\chi Q \nabla_h f \|_{\vecH_{sc}^{s+2}(\R^3)} \lesssim h \|\vecu\|_{\vecH_{sc}^{s+1}(\R^3)}.
\end{eqnarray*}
Recall that $\vecu =  Q \nabla_h \underline{f}+hK_{-M}\vecu$. Then for $h$ small enough
\begin{eqnarray*}
 \|\vecu\|_{\vecH_{sc}^{s+1}(\R^3)} \lesssim  \|f\|_{\overline{H}_{sc}^s(D)},
\end{eqnarray*}
and therefore
\begin{eqnarray*}
\|\chi Q \nabla_h f \|_{\vecH_{sc}^{s+2}(\R^3)} \lesssim h \|f\|_{\overline{H}_{sc}^s(D)}.
\end{eqnarray*}
From the inequality (\ref{PrelinimaryResultsTrace}) we have that
\begin{eqnarray*}
|\gamma Q \nabla_h f |_{\vecH^{s+\frac{3}{2}}(\Gamma)}=|\gamma (\chi Q \nabla_h f) |_{\vecH^{s+\frac{3}{2}}(\Gamma)} \lesssim h^{\frac{1}{2}} \|f\|_{\overline{H}_{sc}^s(D)}.
\end{eqnarray*}
This completes the proof.
\end{proof} \proofend

\section*{\normalsize Acknowledgments}
We thank the anonymous referees for their careful reading of our manuscript which help to improve the readability and quality of the paper. The research of S. Meng  is supported in part by the Chateaubriand STEM
fellowship. S. Meng greatly acknowledges the hospitality of the DeFI Team at
INRIA and Ecole Polytechnique.


\section*{\normalsize References}
\begin{thebibliography}{00}
\bibitem{A} S. Agmon,
\newblock { \it Lectures on elliptic boundary value problems}.
\newblock VanNostrand Mathematical Studies 2, 1965.


\bibitem{AG} S. Alinhac and P. G\'{e}rard,
\newblock { \it Pseudo-differential Operators and the Nash-Moser Theorem}.
\newblock American Mathematical Society, 2007.


\bibitem{ABDG} C. Amrouche, C. Bernardi, M. Dauge and V. Girault, Vector potentials in three-dimensional non-smooth domains, {\it Mathematical Methods in the Applied Sciences}, {\bf 21}, 9, 823-864, (1998).


\bibitem{BPP2014}
    E. {Bl{\aa}sten, L. P{\"a}iv{\"a}rinta and
              J. Sylvester,}
     \newblock { \it Corners always scatter}.
   {Comm. Math. Phys.},
   {\bf 331}, 2, 725-753
      (2014).
    

\bibitem{CaCo} F. Cakoni  and  D. Colton,
\newblock {\it A Qualitative Approach to Inverse Scattering Theory}.
\newblock Springer, Berlin 2014.


\bibitem{CaCoMo} F. Cakoni, D. Colton and P. Monk, {\it The Linear Sampling Method in Inverse Electromagnetic Scattering} CBMS-NSF, {\bf 80}, SIAM Publications, 2011.


\bibitem{CaGHa}
F. Cakoni, D. Gintides  and H. Haddar,
\newblock{The existence of an infinite discrete set of transmission eigenvalues},
{\em SIAM Jour. Math. Anal.} {\bf 42},  237--255,  (2010).


\bibitem{CaHa1}
F. Cakoni and H. Haddar, A variational approach for the solution of the electromagnetic interior transmission problem for anisotropic media, {\it Inverse Problems and Imaging} {\bf 1},  443-456 (2007).


\bibitem{CaHa2}
F. Cakoni and H. Haddar,
\newblock Transmission eigenvalues in inverse scattering theory,
\newblock {\it Inverse Problems and Applications, Inside Out}  {\bf 60}, MSRI Publications 2012.


 \bibitem{CaHa3} F. Cakoni and H. Haddar, Transmission eigenvalues,  {\it Inverse Problems}, {\bf 29} 100201, (2013).

\bibitem{cchlsm}
{F. Cakoni, D. Colton and H. Haddar},
\newblock{On the determination of Dirichlet or transmission eigenvalues from far field data},
{\it C. R. Acad. Sci. Paris},

{\bf 348},
{379-383} ({2010}).


\bibitem{CaHaMe} F. Cakoni, H. Haddar and S. Meng, Boundary Integral Equations for the Transmission Eigenvalue Problem for Maxwell Equations, {\it J. Integral Equations and Applications}, to appear, 2015.


\bibitem{Ch}
L. Chesnel, Interior transmission eigenvalue problem for Maxwell's equations: the T-coercivity as an alternative approach, {\it Inverse Problems} {\bf 28}, no. 6, 065005 (2012). 

\bibitem{CHL} S. Chanillo, B. Helffer and A. Laptev, Nonlinear eigenvalues and analytic hypoellipticity, {\it Journal of Functional Analysis}  {\bf 209} (2004) 425-443.

\bibitem{CK}
D.~Colton and R.~Kress,
\newblock {\em Inverse Acoustic and Electromagnetic Scattering Theory}.
\newblock Springer, New York, 3nd edition 2013.

\bibitem{CoHa}
A. Cossonni\`ere  and H. Haddar, The electromagnetic interior transmission problem for regions with cavities. {\it SIAM J. Math. Anal.}, {\bf 43}, 1698-1715 (2011).

\bibitem{CoH} A. Cossonni\`{e}re and H. Haddar, Surface integral formulation of the interior transmission problem, {\it J. Integral Equations and Applications} 25 (2013), 1123-1138.

\bibitem{DP} M. Dimassi and V. Petkov, Upper bound for the counting function of interior transmission
eigenvalues, preprint 2013, arXiv: math.SP: 1308.2594v4.

\bibitem{H}  H. Haddar, The interior transmission problem for anisotropic Maxwell's equations and its applications to the inverse problem, {\it Math. Methods Appl. Sci.} {\bf 27}, no. 18, 2111-2129 (2004).

\bibitem{HiKrOlPa2} M. Hitrik, K. Krupchyk, P. Ola, and L. P\"aiv\"arinta, Transmission eigenvalues for elliptic operators, {\it SIAM J. Math. Anal.}, { 43}, 2630-2639 (2011).

\bibitem{KH} A. Kirsch and F. Hettlich, {\em The Mathematical Theory of Time-Harmonic  Maxwell's Equations}, Vol 190, Springer, Berlin 2015.

\bibitem{K}
 A. Kirsch,  On the existence of transmission eigenvalues, {\it Inverse Problems and  Imaging}, {\bf 2}, 155-172, (2009).

\bibitem{LV1} Lakshtanov E., Vainberg B., Ellipticity in the interior transmission problem in anisotropic media. {\it SIAM J. Math. Anal.} {\bf 44}, 1165-1174, (2012).

\bibitem{LV2} Lakshtanov E., Vainberg B., Applications of elliptic operator theory to the isotropic interior transmission eigenvalue problem  {\it Inverse Problems}, {\bf 29}, 104003, (2013).

\bibitem{LR2015}
    {A. Lechleiter and M. Rennoch},
      {\it Inside-outside duality and the determination of
              electromagnetic interior transmission eigenvalues},
   {SIAM J. Math. Anal.},
  {\bf 47}, 1, 684-705, 
      (2015).
   
\bibitem{N} J. C. N{\'e}d{\'e}lec, {\it Acoustic and electromagnetic equations. Integral representations for harmonic problems}, Springer-Verlag, New York, (2001).

\bibitem{PV} V. Petkov and G. Vodev, Asymptotics of the number of the interior transmission eigenvalues, {\it J. Spectral Theory}, to appear.

\bibitem{PS}
 L. P\"aiv\"arinta and J. Sylvester,  {Transmission eigenvalues}, {\em SIAM J. Math. Anal.},  {\bf 40} 738-753, (2008).

\bibitem{R1} L. Robbiano, Spectral analysis of the interior transmission eigenvalue problem, {\it  Inverse Problems}, {\bf 29}, 104001, (2013).   

\bibitem{R2} L. Robbiano, Counting function for interior transmission eigenvalues, preprint 2013,
arXiv: math.AP: 1310.6273.

\bibitem{Robert} D. Robert, Non-linear eigenvalue problems, {\it Math. Contemp.} 26, 109–27, (2004).

\bibitem{S} J. Sylvester, Discreteness of transmission eigenvalues via upper triangular compact operator, {\it SIAM J. Math. Anal.} 44:341-354 (2012).

\bibitem{V} G. Vodev, Transmission eigenvalue-free regions, {\it Comm. Math. Phys.}, to appear.

\bibitem{Z} M. Zworski,
\newblock {\it Semiclassical Analysis}.
\newblock American Mathematical Society, 2012.
\end{thebibliography}
\end{document}